\documentclass[a4paper,twoside,10pt]{amsart}
\usepackage{amsmath,amsfonts,amssymb,amsthm}
\newtheorem{theorem}{Theorem}[section]
\newtheorem{lemma}{Lemma}[section]

\newtheorem{corollary}{Corollary}[section]
\usepackage{color,graphicx}
\usepackage[a4paper]{geometry}
\usepackage[colorlinks=true,pdfstartview=FitH,pdfview=FitH]{hyperref}
\numberwithin{equation}{section}
\def\im{\mathrm{i}}
\def\d{\mathrm{d}}
\def\e{\mathrm{e}}
\def\O{{\mathcal O}}
\def\eps{\varepsilon}
\def\hyperF#1#2#3#4{{F}\!\left({#1, #2 \atop #3};#4\right)}
\def\hhyperF#1#2#3#4{{F}\!\left(#1, #2;#3;#4\right)}
\def\hyperOlverF#1#2#3#4{{\mathbf{F}}\!\left({#1, #2 \atop #3};#4\right)}

\author[G. Nemes]{Gerg\H{o} Nemes}
\email{nemes.gergo@renyi.hu}
\address{Alfr\'ed R\'enyi Institute of Mathematics, Re\'altanoda utca 13--15, Budapest H-1053, Hungary}
\author[A. B. Olde Daalhuis]{Adri B. Olde Daalhuis}
\email{A.OldeDaalhuis@ed.ac.uk}
\address{School of Mathematics and Maxwell Institute for Mathematical Sciences, The University of Edinburgh, Edinburgh EH9 3FD, United Kingdom}


\begin{document}

\title[Large-parameter asymptotics for the Legendre and allied functions]{Large-parameter asymptotic expansions\\ for the Legendre and allied functions}

\begin{abstract} 
Surprisingly, apart from some special cases, simple asymptotic expansions for the associated Legendre functions
 $P_\nu ^\mu (z)$ and $Q_\nu ^\mu (z)$ for large degree $\nu$ or large order $\mu$ are not available in the literature. 
The main purpose of the present paper is to fill this gap by deriving simple (inverse) factorial expansions for these functions and provide sharp and realistic bounds on their error terms. Analogous results for the Ferrers functions and the closely related Gegenbauer function are also included. In the cases that $\nu$ is an integer or $2\mu$ is an odd integer, many of these new expansions terminate and provide finite representations in terms of simple functions. Most of these representations appear to be new. It is well known that the hypergeometric series can be regarded as a large-$c$ asymptotic expansion for the hypergeometric function $\hhyperF{a}{b}{c}{z}$. We also derive computable bounds for the remainder term of this expansion. To our best knowledge, no such estimates have been given in the literature prior to this paper.
\end{abstract}
\maketitle

\section{Introduction and main results}\label{SecIntro}

In this paper, we derive asymptotic expansions for the associated Legendre functions $P_\nu ^\mu  (z)$ and $Q_\nu ^\mu  (z)$, 
and the Ferrers functions $\mathsf{P}_\nu ^\mu  (x)$ and $\mathsf{Q}_\nu ^\mu  (x)$ in the case that either the degree $\nu$ or the order $\mu$ becomes large. These functions are solutions of the associated Legendre differential equation
\begin{equation}\label{LegendreODE}
(1 - z^2 )\frac{\d^2 w(z)}{\d z^2} - 2z\frac{\d w(z)}{\d z} + \left( \nu (\nu  + 1) - \frac{\mu ^2}{1 - z^2 } \right)w(z) = 0.
\end{equation}
The equation \eqref{LegendreODE} occurs in the theory of spherical harmonics, potential theory, quantum mechanics and in other branches of applied mathematics (see, for example, \cite[\href{http://dlmf.nist.gov/14}{Ch.\ 14}]{NIST:DLMF}).
Over the past several decades, a great deal of effort has gone into deriving asymptotic expansions for these functions. 
 Most of these (uniform) expansions are in terms of non-elementary functions (Airy functions, Bessel functions, parabolic cylinder functions, etc.) and often 
 have complicated coefficients (see the end of this section for a short summary). In contrast, the expansions we present in this paper are just (inverse) 
 factorial expansions, with simple, elementary coefficients. 
 In addition, we provide, in Sections \ref{SectHypergeometric}--\ref{GegenbauerBounds}, sharp and realistic error bounds for the new expansions. 
 Numerical examples demonstrating the sharpness of these bounds and the accuracy of the asymptotic expansions are given in Section \ref{Numerics}.
 Note, however, that unlike their uniform counterparts, our expansions break down near the transition points $z=\pm 1$ of the 
 equation \eqref{LegendreODE}, except for the large-$\mu$ asymptotic expansions for the associated Legendre functions which require $z$ to be bounded instead. We find it surprising that, apart from some special cases, these simple expansions have not been given in the literature before. 
 The reason might be that previous authors were mainly focusing on deriving asymptotic results that are uniform in character, and hence, 
 inevitably posses a rather complicated structure. Some of our results are special cases of the well-known large-$c$ asymptotic expansion of the hypergeometric function $\hhyperF{a}{b}{c}{z}$, for which we also provide, for the first time, computable error bounds.

A common feature of our large-$\nu$ expansions below is that they terminate when $2\mu$ equals an odd integer, thereby giving simple, closed-form expressions for the associated Legendre functions and the Ferrers functions. Most of these representations appear to be new. Similarly, the large-$\mu$ expansions terminate in the case that $\nu$ is an integer and, again, yield simple, finite representations.

We begin with the definition of the functions under consideration. The associated Legendre functions of complex degree $\nu$ and order $\mu$ are defined in terms of the (regularised) hypergeometric function as follows (cf. \cite[\href{http://dlmf.nist.gov/14.3.ii}{\S 14.3(ii)}]{NIST:DLMF}):
\[
P_\nu ^\mu  (z) = \left(\frac{z + 1}{z - 1}\right)^{\mu /2}\hyperOlverF{\nu  + 1}{- \nu }{1 - \mu}{\frac{1-z}{2}}
\]
for $z \in \mathbb{C}\setminus \left(-\infty,1\right]$, and
\begin{equation}\label{QhypergeoRepr1}
 \e^{-\pi \im\mu }Q_\nu ^\mu  (z) = \frac{\pi ^{1/2} \Gamma (\nu  +\mu  +  1)\left(z^2  - 1\right)^{\mu /2} }{2^{\nu  + 1} z^{\nu  +\mu  + 1}}
\hyperOlverF{\frac{\nu  +\mu  + 2}{2}}{\frac{\nu  +\mu  + 1}{2}}{\nu  + \frac{3}{2}}{\frac{1}{z^2}}
\end{equation}
for $z \in \mathbb{C}\setminus \left(-\infty,1\right]$ and $\nu  + \mu  \in \mathbb{C} \setminus \mathbb{Z}_{ < 0}$. For analytic continuation to other Riemann sheets, see \cite[\href{http://dlmf.nist.gov/14.24}{\S 14.24}]{NIST:DLMF}. The Ferrers functions of complex degree $\nu$ and order $\mu$ are defined as follows (cf. \cite[\href{http://dlmf.nist.gov/14.3.i}{\S 14.3(i)}]{NIST:DLMF}):
\begin{equation}\label{PFerrersdef}
\mathsf{P}_\nu ^\mu  (x) = \left( \frac{1 + x}{1 - x} \right)^{\mu /2} \hyperOlverF{\nu  + 1}{ - \nu}{1 - \mu }{\frac{1-x}{2}}
\end{equation}
for $-1<x<1$, and
\begin{multline*}
\mathsf{Q}_\nu ^\mu  (x) = \frac{\pi }{2\sin (\pi \mu )}\left( \cos (\pi \mu )\left( \frac{1 + x}{1 - x}\right)^{\mu /2} \hyperOlverF{\nu  + 1}{- \nu}{1 - \mu}{\frac{1-x}{2}}\right.\\
\left.- \frac{\Gamma (\nu  + \mu  + 1)}{\Gamma (\nu  - \mu  + 1)}\left( \frac{1 - x}{1 + x} \right)^{\mu /2} \hyperOlverF{\nu  + 1}{- \nu}{1 + \mu}{\frac{1-x}{2}}\right)
\end{multline*}
for $-1<x<1$ and $\nu  + \mu  \in \mathbb{C} \setminus \mathbb{Z}_{ < 0}$. When $\mu$ is an integer, the right-hand side is replaced by its limiting value.

We shall also study the closely related Gegenbauer function. The Gegenbauer function of complex degree $\nu$ and order $\lambda$ may be defined as follows (cf. \cite[\href{http://dlmf.nist.gov/14.3.E21}{Eq.\ 14.3.21}]{NIST:DLMF} and \cite[\href{http://dlmf.nist.gov/14.3.E22}{Eq.\ 14.3.22}]{NIST:DLMF}):
\begin{equation}\label{eq69}
C_\nu ^{(\lambda )} (z) = \frac{\sqrt\pi  \Gamma (\nu  + 2\lambda )}{\Gamma (\lambda )\Gamma (\nu  + 1)\left(2\sqrt{z^2-1}\right)^{\lambda  - \frac{1}{2}} }
P_{\nu  + \lambda  - \frac{1}{2}}^{\frac{1}{2} - \lambda } (z)
\end{equation}
for $z \in \mathbb{C}\setminus \left(-\infty,1\right]$, and
\begin{equation}\label{eq66}
C_\nu ^{(\lambda )} (x) = \frac{\sqrt\pi \Gamma (\nu  + 2\lambda )}{\Gamma (\lambda )\Gamma (\nu  + 1)\left(2\sqrt{1-x^2}\right)^{\lambda  - \frac{1}{2}} }
\mathsf{P}_{\nu  + \lambda  - \frac{1}{2}}^{\frac{1}{2} - \lambda } (x)
\end{equation}
for $-1<x<1$.

Our initial focus is on the large-$\nu$ asymptotics. In the case that $\nu$ is large, $Q_\nu ^\mu  (z)$ is the recessive solution of differential equation \eqref{LegendreODE}. Our main tools are the following convenient representations for this function.

(1) The associated Legendre function of the second kind can be represented in terms of the modified Bessel function of the second kind as
\begin{equation}\label{QIntRepr1}
 \e^{-\pi \im\mu }Q_\nu ^\mu  (\cosh \xi ) = \frac{1}{\Gamma (\nu  - \mu  + 1)}\int_0^{ + \infty } t^\nu  \e^{ - t\cosh \xi } K_\mu  (t\sinh \xi )\d t ,
\end{equation}
provided that $\xi>0$ and $\Re\nu  > |\Re\mu | - 1$ (compare \cite[Eq.\ (6.8.29)]{Erdelyi1954} and \eqref{QhypergeoRepr1}). An immediate consequence of the factor $t^\nu$ in the integrand in \eqref{QIntRepr1} is that when $\Re\nu$ is large and positive, the main contribution to this integral comes from large values of $t$, and thus we can obtain a large-$\nu$ asymptotic expansion by expanding the Bessel function by its well-known large-$t$ asymptotic expansion. In this way it is also straightforward to obtain sharp and realistic error bounds.

(2) If combined with a quadratic transformation \cite[\href{http://dlmf.nist.gov/15.8.E19}{Eq.\ 15.8.19}]{NIST:DLMF} for the hypergeometric function, \eqref{QhypergeoRepr1} yields the representation
\begin{equation}\label{QhypergeoRepr2}
\e^{-\pi\im\mu}Q_\nu ^\mu  (\cosh\xi)= \sqrt{\frac{\pi}{2\sinh\xi}}\e^{-\left(\nu+\frac12\right)\xi}\Gamma(\nu+\mu+1)
\hyperOlverF{\frac12+\mu}{\frac12-\mu}{\nu+\frac32}{\frac{-\e^{-\xi}}{2\sinh\xi}},
\end{equation}
(cf. \cite[Eq.\ (3.2.44)]{HTF1})
in which, again,  $\nu  + \mu  \in \mathbb{C} \setminus \mathbb{Z}_{ < 0}$ and we take $\xi\in\mathcal{D}_1$ where 
\begin{equation}\label{RegionD}
\mathcal{D}_p=\left\{\xi : \Re \xi>0, |\Im \xi |<\pi p\right\}, \quad p>0.
\end{equation} 
Note that the function $\cosh \xi$ is a biholomorphic bijection between $\mathcal{D}_1$ and $\mathbb{C}\setminus \left(-\infty,1\right]$.

Any hypergeometric function for which a quadratic transformation exists can be expressed in terms of the associated Legendre functions or the Ferrers functions (cf. \cite[\href{https://dlmf.nist.gov/15.9.iv}{\S 15.9(iv)}]{NIST:DLMF}). Consequently, there are several representations in terms of hypergeometric functions. For large parameter asymptotic approximations of hypergeometric functions see, e.g., \cite{Watson1918} and \cite{KwOD14}. These results can be used to obtain asymptotic expansions for the associated Legendre functions or the Ferrers functions. In general, the hypergeometric series
\[
\hyperF{a}{b}{c}{z}=\Gamma(c)\hyperOlverF{a}{b}{c}{z}=\sum_{n=0}^\infty\frac{\left(a\right)_n \left(b\right)_n}{\left(c\right)_n n!} z^n
\]
converges only for $|z|<1$, but it can be regarded as a large-$c$ asymptotic series in much larger complex $z$-domains 
(cf. \cite[\href{https://dlmf.nist.gov/15.12.ii}{\S 15.12(ii)}]{NIST:DLMF}). Thus, the representation \eqref{QhypergeoRepr2} is very 
convenient since it is possible to obtain a large-$\nu$ asymptotic expansion by replacing the hypergeometric function with its hypergeometric series. This asymptotic expansion will be valid in the sector $|\arg\nu|\leq\pi-\delta$ ($<\pi$) and for $\xi\in\mathcal{D}_1$, $\left| \sinh \xi \right| \ge \eps$ ($> 0$). Error bounds for this expansion will then follow from the more general results we shall prove for the hypergeometric series (see Section \ref{SectHypergeometric}).

Before stating our main results, we introduce some notation. We define, for any $n\geq 0$,
\[
a_n (\mu ) = \frac{(4\mu ^2  - 1^2 )(4\mu ^2  - 3^2 ) \cdots (4\mu ^2  - (2n - 1)^2 )}{8^n n!}
=\frac{\left(\frac12-\mu\right)_n\left(\frac12+\mu\right)_n}{\left(-2\right)^n n!},
\]
and let
\begin{equation}\label{eq6}
C(\xi ,\mu ) = \begin{cases} \sin (\pi \mu ) & \text{if }\; \xi > 0, \\ \pm \im \e^{ \mp \pi \im\mu } & \text{if } \; 0<\pm \Im\xi  < \pi, \end{cases}
\qquad \chi(p)= \frac{\sqrt\pi\Gamma\left(\frac{p}{2}+1\right)}{\Gamma\left(\frac{p}{2}+\frac{1}{2}\right)}, \quad p>0.
\end{equation}

The main large parameter asymptotic expansions we prove in this paper are as follows.

\subsection*{Legendre functions for large \texorpdfstring{$\nu$}{nu} and fixed \texorpdfstring{$\mu$}{mu}.} Assume that $\xi\in\mathcal{D}_1$, $\left| \sinh \xi \right| \ge \eps$ ($> 0$) and $\mu\in \mathbb{C}$ is bounded. Then the associated Legendre functions have the inverse factorial expansions
\begin{gather}\label{formal1}
\begin{split}
 P_\nu ^\mu  (\cosh \xi )\sim \; & \frac{\e^{ \left(\nu+\frac{1}{2} \right)\xi}}{\sqrt{2\pi \sinh \xi}} \sum_{n = 0}^{\infty} a_n (\mu ) 
\frac{\Gamma \left( \nu  - n + \tfrac{1}{2} \right)}{\Gamma (\nu  - \mu  + 1)} \left( \frac{-\e^{ - \xi } }{\sinh \xi } \right)^n \\
& + \frac{C(\xi ,\mu ) \e^{- \left(\nu+\frac{1}{2} \right)\xi}}{\sqrt{2\pi \sinh \xi}} \sum_{n = 0}^{\infty} a_n (\mu )  
\frac{\Gamma \left( \nu  - n + \tfrac{1}{2} \right)}{\Gamma (\nu  - \mu  + 1)}\left( \frac{\e^\xi }{\sinh \xi} \right)^n,
\end{split}
\end{gather}
\begin{equation}\label{formal2}
 \e^{-\pi \im\mu }Q_\nu ^\mu  (\cosh \xi ) \sim \sqrt{\frac{\pi}{2\sinh \xi } } \e^{ -  \left(\nu+\frac{1}{2}\right) \xi}
\sum_{n = 0}^{\infty} a_n (\mu )\frac{\Gamma \left( \nu  - n + \tfrac{1}{2} \right)}{\Gamma (\nu  - \mu  + 1)}\left( \frac{\e^\xi }{\sinh \xi } \right)^n,
\end{equation}
as $\Re\nu\to+\infty$, with $\Im \nu$ being bounded. These results are direct consequences of Theorems \ref{thm1} and \ref{thm2}, where we give sharp and realistic error bounds for the truncated versions of the expansions. We remark that in the special case that $\xi$ is positive, $\nu$ is a positive integer and $\mu=0$, these expansions (without error estimates) were also given by Olver \cite{Olver2000}. For more information regarding inverse factorial expansions, see, e.g., \cite{OD04a}.

Under the same assumptions, the associated Legendre function of the second kind has the factorial expansion
\begin{equation}\label{formal3}
\e^{-\pi\im\mu}Q_\nu ^\mu  (\cosh\xi)\sim \sqrt{\frac{\pi}{2\sinh\xi}}\e^{-\left(\nu+\frac12 \right)\xi}
\sum_{n=0}^{\infty} a_n(\mu)\frac{\Gamma(\nu+\mu+1)}{\Gamma\left(\nu+\frac32+n\right)}\left(\frac{\e^{-\xi}}{\sinh\xi}\right)^n,
\end{equation}
as $|\nu|\to+\infty$ in the sector $|\arg\nu|\leq\pi-\delta$ ($<\pi$). This result is a direct consequence of Theorems \ref{thm211} and \ref{thm21}, where we provide sharp and realistic error bounds for the truncated version of the expansion. For more information regarding factorial series, see, e.g., \cite{Delabaere2007}. Note that Theorem \ref{thm21} implies \eqref{formal3} for bounded values of $\mu$ satisfying $- \frac{1}{2} < \Re \mu$. This restriction can be removed by appealing to the connection formula \eqref{qmurefl}.

We observe that when $\e^{2\Re\xi}>2\cos\left(2\Im\xi\right)$, the infinite series on the right-hand side of \eqref{formal3} converges to its left-hand side. This is no surprise because this infinite series is just the power series for the hypergeometric function in \eqref{QhypergeoRepr2}. Note that when $\e^{-2\Re\xi}>2\cos\left(2\Im\xi\right)$ and $\nu+\frac{1}{2}\notin \mathbb{Z}$, the infinite series on the right-hand side of \eqref{formal2} is also convergent, however, its sum is not the associated Legendre function on the left-hand side.

We mention that if $2\mu$ equals an odd integer, then the expansions \eqref{formal1}, \eqref{formal2} and \eqref{formal3} terminate and 
represent the corresponding function exactly. This generalises the cases $2\mu=\pm1$ given in \cite[\href{https://dlmf.nist.gov/14.5.iii}{\S14.5(iii)}]{NIST:DLMF}. The exact representations resulting from \eqref{formal3} were also found earlier by Cohl et al. \cite[Eq.\ (5.5)]{Cohl2011}.

To obtain a large-$\nu$ asymptotic expansion for the associated Legendre function $P_\nu ^\mu  (z)$ that holds in the sectors ($0<$) $\delta \leq \left| \arg \nu \right| \leq \pi-\delta$ ($<\pi$), we can combine \eqref{formal3} with the connection formula
\[
\e^{-\pi \im\mu }\cos (\pi \nu )P_\nu ^\mu  (\cosh \xi ) = \frac{Q_{ - \nu  - 1}^{ - \mu } (\cosh \xi ) - Q_\nu ^{ - \mu } (\cosh \xi )}{\Gamma (\nu  - \mu  + 1)\Gamma ( - \nu  - \mu )}
\]
(combine \cite[\href{http://dlmf.nist.gov/14.9.E12}{Eq.\ 14.9.12}]{NIST:DLMF} with \cite[\href{http://dlmf.nist.gov/14.3.E10}{Eq.\ 14.3.10}]{NIST:DLMF}). 

\subsection*{Legendre functions for large \texorpdfstring{$\mu$}{mu} and fixed \texorpdfstring{$\nu$}{nu}.} To obtain large-$\mu$ asymptotic expansions for the associated Legendre functions, we can combine the results above with the identities
\begin{align}
P_\nu ^\mu  (\coth \xi ) & = \sqrt{\frac{2 \sinh \xi}{\pi}} \left( \sin (\pi \mu )\Gamma (\nu  + \mu  + 1)P_{\mu  - \frac{1}{2}}^{ - \nu  - \frac{1}{2}} (\cosh \xi )\right.\label{Whipple1}\\
&\qquad\qquad\qquad\qquad\qquad\qquad -\left.\frac{\im\e^{ - \pi \im\nu } }{\Gamma (\nu  - \mu  + 1)}Q_{\mu  - \frac{1}{2}}^{\nu  + \frac{1}{2}} (\cosh \xi ) \right),\nonumber\\
Q_\nu ^\mu  (\coth \xi ) & = \sqrt{\frac{\pi \sinh \xi}{2} } \e^{\pi \im\mu } \Gamma (\nu+\mu  + 1)P_{\mu  - \frac{1}{2}}^{ - \nu  - \frac{1}{2}} (\cosh \xi ).\label{Whipple2}
\end{align}
These identities follow from Whipple's formulae \cite[\href{http://dlmf.nist.gov/14.9.iv}{\S 14.9(iv)}]{NIST:DLMF} and the equations 
\cite[\href{http://dlmf.nist.gov/14.3.E10}{Eq.\ 14.3.10}]{NIST:DLMF}, \eqref{eqpref} and \eqref{eqqref}. For example, if $\xi\in\mathcal{D}_{\frac12}$, $\left| \sinh \xi \right| \ge \eps$ ($> 0$) and $\nu\in \mathbb{C}$ is bounded, \eqref{formal1} and \eqref{formal2} yield the inverse factorial expansions
\begin{gather}\label{formal4}
\begin{split}
P_\nu ^\mu  (\coth \xi ) \sim \; & \frac{\sin(\pi\mu)}{\pi} \e^{\mu\xi} \sum_{n = 0}^{\infty} a_n\! \left(\nu+\tfrac12\right) \Gamma \left( \mu-n\right) 
 \left( \frac{-\e^{ - \xi } }{\sinh \xi } \right)^n  \\
 & +C\left(\xi,-\mu-\tfrac12\right)\frac{\sin(\pi\nu)}{\pi}\e^{-\mu\xi}\sum_{n = 0}^{\infty} a_n\! \left(\nu+\tfrac12\right) \Gamma \left( \mu-n\right) 
 \left( \frac{\e^{\xi } }{\sinh \xi } \right)^n,
\end{split}
\end{gather}
\begin{gather}\label{formal5}
\begin{split}
 \e^{-\pi\im\mu}Q_\nu ^\mu  (\coth \xi ) \sim \; & \tfrac12 \e^{\mu\xi} \sum_{n = 0}^{\infty} a_n\! \left(\nu+\tfrac12\right) \Gamma \left( \mu-n\right) 
 \left( \frac{-\e^{ - \xi } }{\sinh \xi } \right)^n  \\
 & +C\left(\xi,-\nu-\tfrac12\right)\tfrac{1}{2}\e^{-\mu\xi}\sum_{n = 0}^{\infty} a_n\! \left(\nu+\tfrac12\right) \Gamma \left( \mu-n\right) 
 \left( \frac{\e^{\xi } }{\sinh \xi } \right)^n, 
\end{split}
\end{gather}
as $\Re \mu\to+\infty$, with $\Im \mu$ being bounded.

Examination of the remainder terms in the limiting case $\Re\xi\to 0$, $\Im\xi\neq 0$ (using Theorems \ref{thm1} and \ref{thm2}), shows that these inverse factorial expansions are actually valid on the boundary $\mp\im\xi\in \left[\arcsin \eps,\frac\pi2 \right)$ as well. Bounds for the error terms of these expansions follow from Theorems \ref{thm1} and \ref{thm2}, the details are left to the reader. Note that the function $\coth \xi$ is a biholomorphic bijection between the domains $\mathcal{D}_{\frac12}$ and $\{ z: \Re z>0, z\not\in(0,1]\}$, see Figure \ref{fig1}. In particular, the domain $\{ z: \Re z<0, \Im z\neq 0\}$ is not covered by the expansions \eqref{formal4} and \eqref{formal5}.

\begin{figure}[htbp]
\centering
\includegraphics[width=0.9\textwidth]{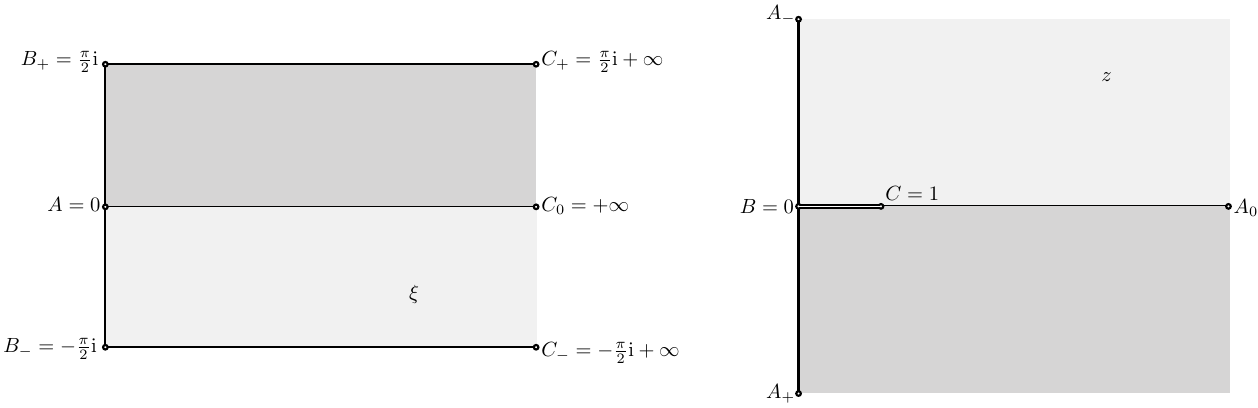}%
\caption{\rm The mapping $z=\coth\xi$ from $\mathcal{D}_{\frac12}$ to $\{ z: \Re z>0, z\not\in(0,1]\}$.
   The boundary $\mp\im\xi\in(0,\frac\pi2)$ is mapped to the imaginary axis $\pm\im z>0$.}
\label{fig1}
\end{figure}

In order to cover the whole Riemann sheet $\mathbb{C}\setminus \left(-\infty,1\right]$, we employ the analytic continuation formulae
\begin{align*}
P_\nu ^\mu \left(z\e^{\pm\pi\im}\right) & = \e^{\pm\pi\im\nu}P_\nu ^\mu (z)-\tfrac{2}{\pi}\sin(\pi(\nu+\mu))\e^{-\pi\im\mu}Q_\nu ^{\mu} (z),\\
Q_\nu ^\mu \left(z\e^{\pm\pi\im}\right) & = -\e^{\mp\pi\im\nu}Q_\nu ^\mu (z),
\end{align*}
(see, e.g., \cite[\href{http://dlmf.nist.gov/14.3.E10}{Eq.\ 14.3.10}]{NIST:DLMF}, \cite[\href{http://dlmf.nist.gov/14.24.E1}{Eq.\ 14.24.1}]{NIST:DLMF} and \cite[\href{http://dlmf.nist.gov/14.24.E2}{Eq.\ 14.24.2}]{NIST:DLMF}) together with the above inverse factorial expansions, to obtain
\begin{gather}\label{formal6}
\begin{split}
P_\nu ^\mu  \left(\e^{\pm\pi\im}\coth \xi \right) \sim \; & -C\left(\xi,\mu+\tfrac12\right)\frac{\sin(\pi\nu)}{\pi} \e^{\mu\xi} 
\sum_{n = 0}^{\infty} a_n\! \left(\nu+\tfrac12\right) \Gamma \left( \mu-n\right) 
 \left( \frac{-\e^{ - \xi } }{\sinh \xi } \right)^n \\
 &+\frac{\sin(\pi\mu)}{\pi}\e^{-\mu\xi}\sum_{n = 0}^{\infty} a_n\! \left(\nu+\tfrac12\right) \Gamma \left( \mu-n\right) 
 \left( \frac{\e^{\xi } }{\sinh \xi } \right)^n,
\end{split}
\end{gather}
\begin{gather}\label{formal7}
\begin{split}
 \e^{-\pi\im\mu}Q_\nu ^\mu  \left(\e^{\pm\pi\im}\coth \xi \right) \sim \; & -C\left(\xi,\nu+\tfrac12\right)\tfrac{1}{2} \e^{\mu\xi} 
 \sum_{n = 0}^{\infty} a_n\! \left(\nu+\tfrac12\right) \Gamma \left( \mu-n\right) 
 \left( \frac{-\e^{ - \xi } }{\sinh \xi } \right)^n \\
 & +\tfrac12 \e^{-\mu\xi}\sum_{n = 0}^{\infty} a_n\! \left(\nu+\tfrac12\right) \Gamma \left( \mu-n\right) 
 \left( \frac{\e^{\xi } }{\sinh \xi } \right)^n,
\end{split}
\end{gather}
as $\Re \mu\to+\infty$, with $\Im \mu$ being bounded and $\pm\Im\xi>0$. 
Again, construction of error bounds is possible and we leave the details to the interested reader. It is seen from Figure \ref{fig1}, for example, that in the case $\xi\in\mathcal{D}_{\frac12}$ and $\Im\xi>0$, we have $\arg \left(\e^{+\pi\im}\coth \xi \right)\in(\frac\pi2,\pi)$.

We note that if $\nu$ is an integer, then the expansions \eqref{formal4}--\eqref{formal7} terminate and represent the corresponding function exactly.

\subsection*{Ferrers functions for large \texorpdfstring{$\nu$}{nu} and fixed \texorpdfstring{$\mu$}{mu}.} Suppose that ($0<$) $\eps\leq\zeta\leq \pi-\eps$ ($<\pi$) and $\mu\in \mathbb{C}$ is bounded, and let
\begin{equation}\label{phasedef1}
\alpha_{\mu,\nu,n}= \left( \nu  - n + \tfrac{1}{2} \right)\zeta  + \left( \mu  + n - \tfrac{1}{2}\right)\tfrac{\pi }{2},\qquad
\beta_{\mu,\nu,n}= \left( \nu  + n + \tfrac{1}{2} \right)\zeta  + \left( \mu  + n - \tfrac{1}{2}\right)\tfrac{\pi }{2}.
\end{equation}
Then the Ferrers functions have the inverse factorial expansions
\begin{align}
\mathsf{P}_\nu ^\mu  (\cos \zeta )\sim \; &\sqrt{\frac{2}{\pi \sin \zeta }}\sum_{n = 0}^{\infty} a_n (\mu ) 
\frac{\Gamma \left( \nu  - n + \tfrac{1}{2} \right)}{\Gamma (\nu  - \mu  + 1)} \frac{\cos (\alpha_{\mu,\nu,n})}{\sin ^n \zeta} ,\label{formal8} \\
\mathsf{Q}_\nu ^\mu  (\cos \zeta ) \sim \; & -\sqrt{\frac{\pi}{2 \sin \zeta }} \sum_{n = 0}^{\infty} a_n (\mu )
\frac{\Gamma \left( \nu  - n + \tfrac{1}{2} \right)}{\Gamma (\nu  - \mu  + 1)}\frac{\sin (\alpha _{\mu ,\nu ,n} )}{\sin ^n \zeta },\label{formal9}
\end{align}
as $\Re\nu\to+\infty$, with $\Im \nu$ being bounded. These results follow from Theorem \ref{thm3}, where we give sharp and realistic error bounds for the truncated versions of the expansions.

Under the same assumptions, the Ferrers functions posses the factorial expansions
\begin{align}
\mathsf{P}_\nu ^\mu  (\cos\zeta)  \sim \; & \sqrt{\frac{2}{\pi\sin\zeta}}\label{formal10}
\sum_{n=0}^{\infty} a_n(\mu)\frac{\Gamma(\nu+\mu+1)}{\Gamma\left(\nu+\frac32+n\right)} \frac{\cos(\beta_{\mu,\nu,n})}{\sin^n\zeta}, \\
\mathsf{Q}_\nu ^\mu  (\cos\zeta) \sim \; & -\sqrt{\frac{\pi}{2\sin\zeta}}\label{formal11}
\sum_{n=0}^{\infty} a_n(\mu)\frac{\Gamma(\nu+\mu+1)}{\Gamma\left(\nu+\frac32+n\right)} \frac{\sin(\beta_{\mu,\nu,n})}{\sin^n\zeta},
\end{align}
as $|\nu|\to+\infty$ in the sector $|\arg\nu|\leq\pi-\delta$ ($<\pi$). These results follow from Theorems \ref{thm231} and \ref{thm23}, where we provide sharp and realistic error bounds for the truncated version of the expansions. Note that Theorem \ref{thm23} implies \eqref{formal10} and \eqref{formal11} for bounded values of $\mu$ satisfying $- \frac{1}{2} < \Re \mu$. This restriction can be removed by appealing to the connection formulae \eqref{pmurefl} and \eqref{qmurefl2}. We remark that the expansion \eqref{formal10} (without error estimates and conditions on $\zeta$, $\mu$ and $\nu$) is posed as an exercise in the book of Whittaker and Watson \cite[Exer.\ 38, p.\ 335]{WW1927} (see also \cite[p.\ 169]{Magnus1966}).

In the case that $\frac\pi6<\zeta<\frac{5\pi}{6}$, the infinite series on the right-hand sides of \eqref{formal10} and \eqref{formal11} converge to the Ferrers functions on the left-hand sides. Note that when $\frac\pi6<\zeta<\frac{5\pi}{6}$ and $\nu+\frac{1}{2}\notin \mathbb{Z}$, the infinite series in \eqref{formal8} and \eqref{formal9} are also convergent, however, their sums are not the Ferrers functions on the left-hand sides.

We mention that if $2\mu$ equals an odd integer, then the expansions \eqref{formal8}--\eqref{formal11} terminate and represent the corresponding 
function exactly. This generalises the cases $2\mu=\pm1$ given in \cite[\href{https://dlmf.nist.gov/14.5.iii}{\S14.5(iii)}]{NIST:DLMF}.

\subsection*{Ferrers functions for large \texorpdfstring{$\mu$}{mu} and fixed \texorpdfstring{$\nu$}{nu}.} Assume that ($0<$) $\eps\leq\zeta\leq \pi-\eps$ ($<\pi$) and $\nu\in \mathbb{C}$ is bounded. Then the Ferrers functions have the inverse factorial expansions
\begin{gather}\label{formal12}
\begin{split}
\mathsf{P}_{\nu} ^{\mu}  (\cos \zeta ) \sim \; & \frac{\sin(\pi\mu)}{\pi}\cot^\mu\left(\tfrac{1}{2}\zeta\right)\sum_{n=0}^\infty a_n\!\left(\nu+\tfrac12\right)\Gamma(\mu-n)
\left(2\sin^2\left(\tfrac{1}{2}\zeta\right)\right)^n\\
&  -\frac{\sin(\pi\nu)}{\pi}\tan^{\mu}\left(\tfrac{1}{2}\zeta\right)\sum_{n=0}^\infty a_n\!\left(\nu+\tfrac12\right)\Gamma(\mu-n)
\left(2\cos^2\left(\tfrac{1}{2}\zeta\right)\right)^n,
\end{split}
\end{gather}
\begin{gather}\label{formal13}
\begin{split}
\mathsf{Q}_{\nu} ^{\mu}  (\cos \zeta ) \sim \; & \frac{\cos(\pi\mu)}{2}\cot^\mu\left(\tfrac{1}{2}\zeta\right)\sum_{n=0}^\infty a_n\!\left(\nu+\tfrac12\right)\Gamma(\mu-n)
\left(2\sin^2\left(\tfrac{1}{2}\zeta\right)\right)^n\\
& - \frac{\cos(\pi\nu)}{2}\tan^\mu\left(\tfrac{1}{2}\zeta\right)\sum_{n=0}^\infty a_n\!\left(\nu+\tfrac12\right)\Gamma(\mu-n)
\left(2\cos^2\left(\tfrac{1}{2}\zeta\right)\right)^n,
\end{split}
\end{gather}
as $\Re \mu\to+\infty$, with $\Im \mu$ being bounded. The proof of these results is given in Section \ref{Ferrerslargemu}. Bounds for the error terms of these expansions follow from Theorems \ref{thm1} and \ref{thm2}, we leave the details to the reader.

We note that if $\nu$ is an integer, then the expansions \eqref{formal12} and \eqref{formal13} terminate and represent the corresponding function exactly.

If combined with Theorems \ref{thm22a} and \ref{thm22}, formula \eqref{PFerrersdef} yields the truncated version, together with error bounds, of the known factorial expansion of the Ferrers function $\mathsf{P}_{\nu} ^{-\mu} (x)$ (see \cite[\href{http://dlmf.nist.gov/14.15.E1}{Eq.\ 14.15.1}]{NIST:DLMF}). The analogous expansions for $\mathsf{P}_{\nu} ^{\mu} (x)$ and $\mathsf{Q}_{\nu} ^{\pm\mu} (x)$ may then be obtained by means of connection formulae, see \cite[\href{https://dlmf.nist.gov/14.15.i}{\S14.15(i)}]{NIST:DLMF}.

\subsection*{Gegenbauer function for large \texorpdfstring{$\nu$}{nu} and fixed \texorpdfstring{$\lambda$}{lambda}.} Suppose that $\xi\in\mathcal{D}_1$, $\left| \sinh \xi \right| \ge \eps$ ($>0$), ($0<$) $\eps\leq\zeta\leq \pi-\eps$ ($<\pi$) and $\lambda\in \mathbb{C}$ is bounded, and let
\begin{equation}\label{phasedef2}
K(\xi ,\lambda ) = \begin{cases} \cos (\pi \lambda ) & \text{if }\; \xi > 0, \\ \e^{ \pm \pi \im\lambda } & \text{if } \; 0<\pm \Im\xi  < \pi, \end{cases}
\qquad
\begin{array}{cc}
\gamma _{\lambda,\nu, n}  &= (\nu  + \lambda  - n)\zeta  - (\lambda  - n)\tfrac{\pi }{2},\\
\delta _{\lambda ,\nu ,n} & = (\nu  + \lambda  + n)\zeta  - (\lambda  - n)\tfrac{\pi}{2}.
\end{array}
\end{equation}
Then the Gegenbauer function has the inverse factorial expansions
\begin{gather}\label{formal14}
\begin{split}
C_\nu ^{(\lambda )} (\cosh \xi ) \sim \; & \frac{\e^{\xi (\nu  + \lambda )} }{\left(2\sinh \xi \right)^\lambda  }\sum_{n = 0}^{\infty} a_n\!\left( \lambda  - \tfrac{1}{2} \right)
\frac{\Gamma (\nu  + \lambda -n)}{\Gamma (\lambda )\Gamma (\nu+1)}\left( \frac{-\e^{ - \xi } }{\sinh \xi } \right)^n\\  
& + K(\xi ,\lambda )\frac{\e^{ - \xi (\nu  + \lambda )} }{\left(2\sinh \xi \right)^\lambda}  \sum_{n = 0}^{\infty} a_n\!\left( \lambda  - \tfrac{1}{2} \right)
\frac{\Gamma (\nu  + \lambda  - n)}{\Gamma (\lambda )\Gamma (\nu  + 1)}\left( \frac{\e^\xi }{\sinh \xi } \right)^n,
\end{split}
\end{gather}
\begin{gather}\label{formal15}
\begin{split}
C_\nu ^{(\lambda )} (\cos \zeta ) \sim \; & \frac{2}{\left(2\sin\zeta\right)^\lambda }\sum_{n = 0}^{\infty}a_n \!\left( \lambda  - \tfrac{1}{2} \right)
\frac{\Gamma (\nu  + \lambda  - n)}{\Gamma (\lambda )\Gamma (\nu  + 1)}\frac{\cos (\gamma _{\lambda,\nu, n} )}{\sin ^n \zeta },
\end{split}
\end{gather}
as $\Re\nu\to+\infty$, with $\Im \nu$ being bounded. These results are direct consequences of Corollaries \ref{gegen1} and \ref{gegen2}, where we provide sharp and realistic error bounds for the truncated versions of the expansions. If $\nu$ is a positive integer, $C_\nu ^{(\lambda )}(\cosh\xi)$ is a polynomial in $\cosh\xi$. Thus, we can extend \eqref{formal14} to the larger set $\widetilde {\mathcal{D}_1} =\left\{\xi : \Re\xi >0, -\pi < \Im \xi \leq\pi \right\}$ (provided $\left| \sinh \xi \right| \ge \eps$ ($>0$)). The function $\cosh\xi$ is a continuous bijection between $\widetilde {\mathcal{D}_1}$ and $\mathbb{C}\setminus \left[-1,1\right]$. We remark that in the special case that $\nu$ is a positive integer and $\lambda$ is real, the expansions \eqref{formal14} and \eqref{formal15} (without error estimates) were also given by Szeg\H{o} \cite[Theorem 8.21.10, pp.\ 196--197]{Szego1975}. If, in addition, $\lambda =\frac{1}{2}$, the expansions \eqref{formal14} and \eqref{formal15} reduce to the generalized Laplace--Heine asymptotic expansion and Darboux's asymptotic expansion for the Legendre polynomials \cite[Theorems 8.21.3 and 8.21.4, pp.\ 194--195]{Szego1975}. To our knowledge, no error bounds for these latter expansions have been given in the literature prior to this paper.

Under the same assumptions, the Gegenbauer function has the factorial expansion
\begin{equation}\label{formal16}
C_\nu ^{(\lambda )} (\cos \zeta ) \sim \frac{2}{\left(2\sin \zeta\right)^\lambda} \sum_{n = 0}^{\infty} a_n\! \left( \lambda  - \tfrac{1}{2} \right)
\frac{\Gamma (\nu  + 2\lambda )}{\Gamma (\lambda )\Gamma \left( \nu  + \lambda  + n + 1 \right)}\frac{\cos (\delta _{\lambda ,\nu ,n} )}{\sin ^n \zeta},
\end{equation}
as $|\nu|\to+\infty$ in the sector $|\arg\nu|\leq\pi-\delta$ ($<\pi$). This result follows from \eqref{eq66} and \eqref{formal10}. In Corollaries \ref{gegen3} and \ref{gegen4}, we give sharp and realistic error bounds for the truncated version of the expansion. We remark that in the special case that $\nu$ is a positive integer and $0<\lambda<1$, the expansion \eqref{formal16} (together with the error estimate \eqref{eq68} below) was also given by Szeg\H{o} \cite[Theorem 8.21.11, p.\ 197]{Szego1975}. If, in addition, $\lambda =\frac{1}{2}$, the expansion reduces to Stieltjes' asymptotic expansion for the Legendre polynomials \cite[Theorem 8.21.5, p.\ 195]{Szego1975}.

In the case that $\frac\pi6<\zeta<\frac{5\pi}{6}$, the infinite series on the right-hand side of \eqref{formal16} converge to the Gegenbauer function on the left-hand side. Note that when $\frac\pi6<\zeta<\frac{5\pi}{6}$ and $\nu+\lambda\notin \mathbb{Z}$, the infinite series in \eqref{formal15} is also convergent, however, its sum is not the Gegenbauer function on the left-hand side.

We mention that if $\lambda$ is an integer, then the expansions \eqref{formal14}--\eqref{formal16} terminate and represent the corresponding function exactly.

\subsection*{The cases in which \texorpdfstring{$-\nu$}{-nu} or \texorpdfstring{$-\mu$}{-mu} is large.} The asymptotic expansions we gave in this section are valid either in half-planes or in sectors not including the negative real axis. For the reader's convenience, we reproduce here a series of known connection relations satisfied by the Legendre functions and the Ferrers functions. By combining the above asymptotic expansions with these connection relations, it is possible to cover larger regions in the parameter spaces. According to \cite[\href{http://dlmf.nist.gov/14.9.iii}{\S 14.9(iii)}]{NIST:DLMF} and \cite[\href{http://dlmf.nist.gov/14.3.E10}{Eq.\ 14.3.10}]{NIST:DLMF}, 
\begin{align}
P_{-\nu} ^\mu  (z) & =  P_{\nu-1} ^\mu  (z ), \label{eqpref} \\
Q_{-\nu} ^\mu  (z) & =  -\e^{\pi\im\mu}\cos(\pi\nu)\Gamma(\nu+\mu)\Gamma(\mu-\nu+1)P_{\nu-1} ^{-\mu}  (z )+Q_{\nu-1} ^\mu  (z ), \label{eqqref} \\
P_\nu ^{-\mu}  (z ) & = \frac{\Gamma(\nu-\mu+1)}{\Gamma(\nu+\mu+1)}\left(P_\nu ^\mu  (z )-\tfrac{2}{\pi}\e^{- \pi\im\mu}\sin(\pi\mu)Q_\nu ^\mu  (z )\right), \nonumber \\ 
Q_\nu ^{-\mu}  (z ) & = \e^{-2 \pi\im\mu} \frac{\Gamma(\nu-\mu+1)}{\Gamma(\nu+\mu+1)} Q_\nu ^\mu  (z ). \label{qmurefl}
\end{align}
Similarly, according to \cite[\href{http://dlmf.nist.gov/14.9.i}{\S 14.9(i)}]{NIST:DLMF},
\begin{align}
\mathsf{P}_{ - \nu }^\mu  (x) & = \mathsf{P}_{\nu  - 1}^\mu  (x), \nonumber \\
\sin (\pi (\nu  - \mu ))\mathsf{Q}_{ - \nu }^\mu  (x) & =  - \pi \cos (\pi \nu )\cos (\pi \mu )\mathsf{P}_{\nu  - 1}^\mu  (x) + \sin (\pi (\nu  + \mu ))\mathsf{Q}_{\nu  - 1}^\mu  (x),\nonumber \\
\mathsf{P}_\nu ^{ - \mu } (x) & = \frac{\Gamma (\nu  - \mu  + 1)}{\Gamma (\nu  + \mu  + 1)}\left( \cos (\pi \mu )\mathsf{P}_\nu ^\mu  (x) - \tfrac{2}{\pi }\sin (\pi \mu )\mathsf{Q}_\nu ^\mu  (x) \right),\label{pmurefl} \\
\mathsf{Q}_\nu ^{ - \mu } (x) & = \frac{\Gamma (\nu  - \mu  + 1)}{\Gamma (\nu  + \mu  + 1)}\left( \tfrac{\pi }{2}\sin (\pi \mu )\mathsf{P}_\nu ^\mu  (x) + \cos (\pi \mu )\mathsf{Q}_\nu ^\mu  (x) \right). \label{qmurefl2}
\end{align}

\subsection*{Earlier asymptotic results.} We close this section by providing a brief summary of existing asymptotic results in the literature. Thorne \cite{Thorne1957} obtained uniform asymptotic expansions for $P_\nu ^\mu  (z)$ and $Q_\nu ^\mu  (z)$ with the assumptions that $z$ is complex, $\mu$ is a negative integer, $\nu$ is a large positive integer and $0<-\mu/\left(\nu+\frac{1}{2}\right)<1$ is fixed. Olver \cite{Olver1975a} derived uniform expansions for $\mathsf{P}_\nu ^\mu (x)$ and $\mathsf{Q}_\nu ^\mu  (x)$ under the conditions that $0 \leq x <1$, $\mu$ is real and negative, $\nu$ is a large positive real number, and $0<-\mu/\left(\nu+\frac{1}{2}\right)\leq 1$ is fixed. He also treated the case when $\mu$ and $\nu+\frac{1}{2}$ are both purely imaginary. Olver \cite[\S12 and \S13]{Olv97} also gave uniform asymptotic expansions with error bounds for $P_\nu ^\mu  (z)$, $Q_\nu ^\mu  (z)$, $\mathsf{P}_\nu ^\mu (x)$ and $\mathsf{Q}_\nu ^\mu  (x)$ with the assumptions that $z$ is complex, $-1< x <1$, $\mu$ is real and fixed and $\nu$ is a large positive real number (see also \cite{Ursell1984}). Similar results were obtained by Shivakumar and Wong \cite{Wong1988} for $P_\nu ^\mu  (z)$ and by Frenzen \cite{Frenzen1990} for $Q_\nu ^\mu (z)$ under the conditions that $z\geq 1$, $\mu<\frac{1}{2}$ is bounded and $\nu$ is a large positive real number. Their expansions are also supplied with computable error bounds. Dunster \cite{Dunster2003} completed the works of Thorne and Olver by deriving uniform asymptotic expansions for the Legendre and Ferrers functions with the assumptions that $\nu+\frac{1}{2}$ is real and positive, $\mu$ is a large negative real number and $-\mu/\left(\nu+\frac{1}{2}\right)>1$. Recently, Cohl et al. \cite{Cohl2018} gave  uniform asymptotic approximations for the Legendre and Ferrers functions under the conditions that $\mu$ is real and bounded, $\nu$ is large and is either real and positive or $\nu+\frac{1}{2}$ is purely imaginary.

\section{Hypergeometric function for large \texorpdfstring{$c$}{c} and fixed \texorpdfstring{$a$}{a}, \texorpdfstring{$b$}{b}: error bounds}\label{SectHypergeometric}

In the following two theorems, we provide computable error bounds for the hypergeometric series. The results in Theorem \ref{thm22a} are especially useful when $\Re c \to+\infty$ and $\Im c$ is bounded, whereas that in Theorem \ref{thm22} is useful when $|c|\to+\infty$ in the larger domain $|\arg c|\leq\pi-\delta$ ($<\pi$).

\begin{theorem}\label{thm22a} Let $N$ be an arbitrary non-negative integer. Let $a$, $b$ and $c$ be complex numbers satisfying $N > \max(-\Re a,-\Re b)$, $\Re c > \Re b$ and $c \notin \mathbb{Z}_{ \leq 0}$. Then
\begin{equation}\label{hypergeomexp}
\hyperF{a}{b}{c}{z}=\Gamma(c)\hyperOlverF{a}{b}{c}{z}=\sum_{n=0}^{N-1}\frac{\left(a\right)_n \left(b\right)_n}{\left(c\right)_n n!} z^n
+R_N^{(F)} (z,a,b,c),
\end{equation}
where the remainder term satisfies the estimate
\begin{equation}\label{est1}
\left| R_N^{(F)} (z,a,b,c) \right| \le \left|\frac{ \Gamma (\Re b + N)\Gamma (c + N)\Gamma (\Re c - \Re b)}{ \Gamma (b + N)\Gamma (\Re c + N)\Gamma (c - b)}\right| \left|\frac{\left(a\right)_N \left(b\right)_N}{\left(c\right)_N N!} z^N\right| \max (1,\e^{(\Im a)\arg (-z)} ),
\end{equation}
provided $\Re z\leq 0$ and with the convention $\left| \arg ( - z)\right| \le \frac{\pi }{2}$. With the extra assumption $\Re a<1$, we also have
\begin{gather}\label{est2}
\begin{split}
\left| R_N^{(F)} (z,a,b,c) \right| \le \; &\left|\frac{ \sin (\pi a) }{ \sin (\pi \Re a)}\right|  \left| \frac{\Gamma (\Re b)\Gamma (c)\Gamma (\Re c - \Re b)}{\Gamma (b)\Gamma (\Re c)\Gamma (c - b)} \right| \\ & \times \left| \frac{(\Re a)_N (\Re b)_N }{(\Re c)_N N!} z^N\right| \times \begin{cases} 1  & \text{ if } \;
\Re z \le  0, \\ \left|\frac{z}{\Im z}\right| & \text{ if } \;
0<\Re z \le \left| z \right|^2, \\  \frac{1}{|1-z|} & \text{ if } \;  \Re z > \left| z \right|^2,\end{cases}
\end{split}
\end{gather}
provided $z\in \mathbb{C}\setminus \left[1,+\infty\right)$. If $\Re a $ is an integer or $\Re b$ is a non-positive integer, then the limiting value has to be taken in this bound. In addition, the remainder term $R_N^{(F)} (z,a,b,c)$ does not exceed the corresponding first neglected term in absolute value and has the same sign provided that $z$ is negative, $a$, $b$ and $c$ are real, and $N > \max(-a,-b)$, $c>b$ and $c \notin \mathbb{Z}_{ \leq 0}$.
\end{theorem}

\begin{theorem}\label{thm22} Let $z$ be a complex number such that $\Re z\leq\frac12$, and $N$ be an arbitrary non-negative integer. Let $a$, $b$ and $c$ be complex numbers satisfying $\Re a>0$, $N >  - \Re b$, $|\arg(c-b)|<\pi$ and $c \notin \mathbb{Z}_{ \leq 0}$. Then the expansion \eqref{hypergeomexp} holds with the remainder estimate
\[
\left| R_N^{(F)} (z,a,b,c) \right| \le A_N \left|\frac{\left(a\right)_N \left(b\right)_N}{\left(c\right)_N N!} z^N\right|.
\]
Here
\[
A_N = A_N(z,a,b,c,\sigma) = \left|\frac{N}{a}\right|
+\left|\frac{\Gamma(\Re b+N)\Gamma(c+N)}{\Gamma(b+N)\Gamma(c-b)}
\frac{\e^{\pi\left|\Im a\right|+|\sigma|\left|\Im b\right|}\left(1+\frac{N}{a}\right)K^{\Re a}}{\left((c-b)\cos(\theta+\sigma)\right)^{\Re b+N}}
\right|,
\]
in which $\theta=\arg(c-b)$, $|\theta+\sigma|<\frac{\pi}{2}$, $|\sigma|\leq\frac{\pi}{2}$ and 
\[
K = K(z,\sigma) = \min\left(\frac{2}{\cos\sigma},\max\left(1,\frac{\sqrt{\left(\Im z\right)^2+\frac14}}{\frac12-\Re z}\right)\right).
\]
In the case that $|\theta|<\frac{\pi}{2}$ we take $\theta+\sigma=0$, that is, $\cos(\theta+\sigma)=1$. Note that $A_N=\O(1)$ as $|c|\to+\infty$.
\end{theorem}

\section{Legendre functions for large \texorpdfstring{$\nu$}{nu} and fixed \texorpdfstring{$\mu$}{mu}: error bounds}\label{SectLegendre}

In the following two theorems, we provide computable error bounds for the inverse factorial expansions \eqref{formal1} and \eqref{formal2} of the associated Legendre functions $P_\nu ^\mu  (z)$ and $Q_\nu ^\mu  (z)$, respectively. For the definition of the domain $\mathcal{D}_1$, see \eqref{RegionD}.

\begin{theorem}\label{thm1} Let $\xi \in \mathcal{D}_1$ and $N$, $M$ be arbitrary non-negative integers. Let $\mu$ and $\nu$ be complex numbers satisfying $\left|\Re\mu \right| <\min \left( N + \frac{1}{2},M + \frac{1}{2} \right)$ and $\Re\nu  > \max \left( N - \frac{1}{2},M - \frac{1}{2} \right)$. Then
\begin{gather}\label{eq3}
\begin{split}
 P_\nu ^\mu  (\cosh \xi ) =\; & \frac{\e^{ \left(\nu+\frac{1}{2} \right)\xi}}{\Gamma (\nu  - \mu  + 1)\sqrt{2\pi \sinh \xi}} \left( \sum_{n = 0}^{N - 1} a_n (\mu ) 
\Gamma \left( \nu  - n + \tfrac{1}{2} \right) \left( \frac{-\e^{ - \xi } }{\sinh \xi } \right)^n   + R_N^{(P1)} (\xi ,\mu ,\nu ) \right)
\\ & + \frac{C(\xi ,\mu ) \e^{- \left(\nu+\frac{1}{2} \right)\xi}}{\Gamma (\nu  - \mu  + 1)\sqrt{2\pi \sinh \xi}} \left( \sum_{m = 0}^{M - 1} a_m (\mu )  
\Gamma \left( \nu  - m + \tfrac{1}{2} \right)\left( \frac{\e^\xi }{\sinh \xi} \right)^m  + R_M^{(P2)} (\xi ,\mu ,\nu ) \right),
\end{split}
\end{gather}
where $C(\xi ,\mu )$ is defined in \eqref{eq6}, and the remainder terms satisfy the estimates
\begin{align*}
\left| R_N^{(P1)} (\xi ,\mu ,\nu ) \right| \le \; & \left|\frac{ \cos (\pi \mu ) }{ \cos (\pi \Re\mu ) }\right|\left| a_N (\Re\mu ) \right|  
\Gamma \left( \Re\nu  - N + \tfrac{1}{2} \right) \left| \frac{\e^{ - \xi } }{\sinh \xi } \right|^N \\ & \times \begin{cases} 1 & \text{if }\; \Re (\e^{2\xi } ) \le 1, \\ 
\min\big(\left| 1 - \e^{ - 2\xi } \right|\left| \csc (2\Im\xi ) \right|,1 + \chi\big(N+\tfrac{1}{2}\big)\big) & \text{if } \; \Re (\e^{2\xi } ) > 1, \end{cases}
\end{align*}
and
\[
\left| R_M^{(P2)} (\xi ,\mu ,\nu ) \right| \le \left|\frac{\cos (\pi \mu )}{\cos (\pi \Re\mu )}\right|\left|a_M (\Re\mu )\right|
\Gamma \left( \Re\nu  - M + \tfrac{1}{2} \right)\left| \frac{\e^\xi}{\sinh \xi } \right|^M .
\]
If $2\Re\mu $ is an odd integer, then the limiting values have to be taken in these bounds. The square roots are defined to be positive for positive real $\xi$ and are defined by continuity elsewhere. In addition, the remainder term $R_M^{(P2)} (\xi ,\mu ,\nu )$ does not exceed the corresponding first neglected term in absolute value and has the same sign provided that $\xi$ is positive, $\mu$ and $\nu$ are real, and $\left| \mu \right| <M + \frac{1}{2}$ and $\nu > M - \frac{1}{2}$.
\end{theorem}

\begin{theorem}\label{thm2} Let $\xi \in \mathcal{D}_1$ and $N$ be an arbitrary non-negative integer. Let $\mu$ and $\nu$ be complex numbers satisfying $\left|\Re\mu \right| < N + \frac{1}{2}$ and $\Re\nu  > N - \frac{1}{2}$. Then
\begin{equation}\label{eq10}
 \e^{-\pi \im\mu }Q_\nu ^\mu  (\cosh \xi ) = \frac{\sqrt{\frac{\pi}{2\sinh \xi } } \e^{ -  \left(\nu+\frac{1}{2}\right) \xi}}{\Gamma (\nu  - \mu  + 1)}
 \left( \sum_{n = 0}^{N - 1} a_n (\mu )\Gamma \left( \nu  - n + \tfrac{1}{2} \right)\left( \frac{\e^\xi }{\sinh \xi } \right)^n   
 + R_N^{(Q)} (\xi ,\mu ,\nu ) \right) ,
\end{equation}
where the remainder term satisfies the estimate
\begin{equation}\label{eq10a}
\left| R_N^{(Q)} (\xi ,\mu ,\nu ) \right| \le \left| \frac{\cos (\pi \mu ) }{\cos (\pi \Re\mu ) }\right|\left| a_N (\Re\mu ) \right| 
\Gamma \left( \Re\nu  - N + \tfrac{1}{2} \right) \left| \frac{\e^\xi }{\sinh \xi } \right|^N.
\end{equation}
If $2\Re\mu $ is an odd integer, then the limiting value has to be taken in this bound. The square root is defined to be positive for positive real $\xi$ and is defined by continuity elsewhere. In addition, the remainder term $R_N^{(Q)} (\xi ,\mu ,\nu )$ does not exceed the corresponding first neglected term in absolute value and has the same sign provided that $\xi$ is positive, $\mu$ and $\nu$ are real, and $\left| \mu \right| <N + \frac{1}{2}$ and 
$\nu > N - \frac{1}{2}$.
\end{theorem}

A combination of Theorems \ref{thm22a} and \ref{thm22} with the representation \eqref{QhypergeoRepr2} leads to error bounds for the factorial expansion \eqref{formal3} of the associated Legendre function $Q_\nu ^\mu  (z)$ as given in the following two theorems.

\begin{theorem}\label{thm211} Let $N$ be an arbitrary non-negative integer. Let $\mu$ and $\nu$ be complex numbers satisfying $\left|\Re \mu \right| < N + \frac{1}{2}$ and $\Re \nu  >  - \Re \mu-1$. Then
\begin{equation}\label{factexpQ}
\e^{-\pi\im\mu}Q_\nu ^\mu  (\cosh\xi)= \sqrt{\frac{\pi}{2\sinh\xi}}\e^{-\left(\nu+\frac12\right)\xi}\Gamma(\nu+\mu+1)
\left(\sum_{n=0}^{N-1} \frac{a_n(\mu)}{\Gamma\left(\nu+\frac32+n\right)}\left(\frac{\e^{-\xi}}{\sinh\xi}\right)^n +\widehat{R}_N^{(Q)}(\xi,\mu,\nu)\right),
\end{equation}
where the remainder term satisfies the estimate
\begin{align*}
\left|\widehat{R}_N^{(Q)}(\xi,\mu,\nu)\right|\leq \; & \left|\frac{ \Gamma \left( \frac{1}{2}- \Re \mu  + N \right) \Gamma \left( \Re \nu +\Re \mu  + 1 \right)}{\Gamma \left( \frac{1}{2} - \mu  + N \right) \Gamma \left( \nu  + \mu  + 1 \right) } \right|\max \left( 1,\e^{ (\Im \mu )\arg (\e^\xi  \sinh \xi )} \right) \\ & \times  \frac{\left|a_N(\mu)\right|}{\Gamma\left(\Re\nu+\frac32+N\right)}\left|\frac{\e^{-\xi}}{\sinh\xi}\right|^N,
\end{align*}
provided $\Re (\e^{2\xi } ) \ge 1$ and with the convention $\left| \arg (\e^\xi  \sinh \xi) \right| \le \frac{\pi }{2}$. With the extra assumption $\Re\mu< \frac12$, we also have 
\begin{align*}
\left|\widehat{R}_N^{(Q)}(\xi,\mu,\nu)\right|\leq \; & \left| \frac{\cos (\pi \mu )}{\cos (\pi \Re \mu )} \right|
\left| \frac{\Gamma \left( \frac{1}{2} - \Re \mu  \right)\Gamma (\Re \nu  + \Re \mu  + 1)}{\Gamma \left( \frac{1}{2} - \mu  \right)\Gamma (\nu  + \mu  + 1)} \right|
\frac{\left| a_N (\Re \mu )  \right|}{\Gamma \left( \Re \nu  + \frac{3}{2} + N \right)}\left| \frac{\e^{ - \xi } }{\sinh \xi }\right|^N \\
& \times \begin{cases} 1  & \text{ if } \; \Re (\e^{2\xi } ) \ge 1, \\
\left| 1 - \e^{ - 2\xi } \right|\left| \csc (2\Im\xi ) \right|  & \text{ if } \; 0<\Re (\e^{2\xi } ) < 1, \\
|1 - \e^{ - 2\xi }| & \text{ if } \;  \sqrt 2 | \cos \Im \xi | < 1,\end{cases}
\end{align*}
provided $\xi \in \mathcal{D}_1$. If $2\Re\mu $ is an odd integer, then the limiting value has to be taken in this bound. The square root is defined to be positive for positive real $\xi$ and is defined by continuity elsewhere. In addition, the remainder term $\widehat{R}_N^{(Q)} (\xi ,\mu ,\nu )$ does not exceed the corresponding first neglected term in absolute value and has the same sign provided that $\xi$ is positive, $\mu$ and $\nu$ are real, and $\left| \mu \right| <N + \frac{1}{2}$ and $\nu  >  - \mu-1$.
\end{theorem}

\begin{theorem}\label{thm21} Let $\xi \in \mathcal{D}_1$ and $N$ be an arbitrary non-negative integer. Let $\mu$ and $\nu$ be complex numbers satisfying $-\frac12<\Re\mu<N+\frac12$ and $|\arg(\nu+\mu+1)|<\pi$. Then the expansion \eqref{factexpQ} holds with the remainder estimate
\begin{equation}\label{factexpQa}
\left|\widehat{R}_N^{(Q)}(\xi,\mu,\nu)\right|\leq B_N\left|\frac{a_N(\mu)}{\Gamma\left(\nu+\frac32+N\right)}\left(\frac{\e^{-\xi}}{\sinh\xi}\right)^N\right|.
\end{equation}
Here
\[
B_N = B_N(\xi,\mu,\nu,\sigma) = \left|\frac{N}{\frac12+\mu}\right|
+\left|\frac{\Gamma\left(N+\frac12-\Re\mu\right)\Gamma\left(\nu+\frac32+N\right)}{\Gamma\left(N+\frac12-\mu\right)\Gamma(\nu+\mu+1)}
\frac{\e^{(\pi+|\sigma|)\left|\Im\mu\right|}\left(1+\frac{N}{\frac12+\mu}\right)L^{\frac12+\Re\mu}}{\left((\nu+\mu+1)\cos(\theta+\sigma)\right)^{N+\frac12-\Re\mu}}
\right|,
\]
in which $\theta=\arg(\nu+\mu+1)$, $|\theta+\sigma|<\frac{\pi}{2}$, $|\sigma|\leq\frac{\pi}{2}$ and 
\[
L = L(\xi,\sigma) = \min\left(\frac{2}{\cos\sigma},\max\left(1,\frac{\sqrt{1+\left(\Im\coth\xi\right)^2}}{\Re\coth\xi}\right)\right).
\]
The fractional powers are taking their principal values. In the case that $|\theta|<\frac{\pi}{2}$ we take $\theta+\sigma=0$, that is, $\cos(\theta+\sigma)=1$. Note that $B_N=\O(1)$ as $|\nu|\to+\infty$. 
\end{theorem}

\section{Ferrers functions for large \texorpdfstring{$\nu$}{nu} and fixed \texorpdfstring{$\mu$}{mu}: error bounds}

In the following theorem, we give computable error bounds for the inverse factorial expansions \eqref{formal8} and \eqref{formal9} of the Ferrers functions.

\begin{theorem}\label{thm3} Let $0<\zeta<\pi$ and $N$ be an arbitrary non-negative integer. Let $\mu$ and $\nu$ be complex numbers satisfying $\left|\Re\mu \right| < N + \frac{1}{2}$ and $\Re\nu  > N - \frac{1}{2}$. Then
\begin{align}
& \mathsf{P}_\nu ^\mu  (\cos \zeta )=\frac{\sqrt{\frac{2}{\pi \sin \zeta }}}{\Gamma (\nu  - \mu  + 1)} \left( \sum\limits_{n = 0}^{N - 1} a_n (\mu ) 
\Gamma \left( \nu  - n + \tfrac{1}{2} \right) \frac{\cos (\alpha_{\mu,\nu,n})}{\sin ^n \zeta}  + R_N^{(\mathsf{P})} (\zeta ,\mu ,\nu ) \right),\label{eq15} \\
& \mathsf{Q}_\nu ^\mu  (\cos \zeta ) = \frac{-\sqrt{\frac{\pi}{2 \sin \zeta }}}{\Gamma (\nu  - \mu  + 1)} \left( \sum\limits_{n = 0}^{N - 1} a_n (\mu )
\Gamma \left( \nu  - n + \tfrac{1}{2} \right)\frac{\sin (\alpha _{\mu ,\nu ,n} )}{\sin ^n \zeta }  + R_N^{(\mathsf{Q})} (\zeta ,\mu ,\nu ) \right),\label{eq21}
\end{align}
where $\alpha_{\mu,\nu,n}$ is defined in \eqref{phasedef1}, and the remainder terms satisfy the estimate
\begin{equation}\label{eq16}
\left| R_N^{(\mathsf{P})} (\zeta ,\mu ,\nu ) \right|,\left| R_N^{(\mathsf{Q})} (\zeta ,\mu ,\nu ) \right| \le \left|\frac{\cos (\pi \mu )}{\cos (\pi \Re\mu )}\right|
\left| a_N (\Re\mu )\right| \Gamma \left( \Re\nu  - N + \tfrac{1}{2} \right) \frac{\cosh ( \Im(\alpha_{\mu,\nu,N}))}{\sin ^N \zeta } .
\end{equation}
With the modified conditions $\left| \Re\mu  \right| < N + \frac{1}{2}$ and $\Re\nu  > N + \frac{1}{2}$, we also have
\begin{gather}\label{eq17}
\begin{split}
\left| R_N^{(\mathsf{P})} (\zeta ,\mu ,\nu ) \right| \le \; & \left|\frac{\cos (\pi \mu )}{\cos (\pi\Re\mu)}\right|\left| a_N (\Re\mu )\right| 
\Gamma \left( \Re\nu  - N + \tfrac{1}{2} \right) \frac{|\cos (\alpha _{\mu ,\nu ,N} )|}{\sin ^N \zeta } \\ 
& +\left|\frac{\cos (\pi \mu )}{\cos (\pi \Re\mu )}\right|\left| a_{N+1} (\Re\mu )\right| 
\Gamma \left( \Re\nu  - N - \tfrac{1}{2} \right) \frac{|\sin (\alpha _{\mu ,\nu ,N} )|}{\sin ^{N+1} \zeta }
\end{split}
\end{gather}
and
\begin{gather}\label{eq23}
\begin{split}
\left| R_N^{(\mathsf{Q})} (\zeta ,\mu ,\nu ) \right| \le \; &\left|\frac{\cos (\pi \mu )}{\cos (\pi \Re\mu )}\right|\left| a_N (\Re\mu )\right| 
\Gamma \left( \Re\nu  - N + \tfrac{1}{2} \right) \frac{|\sin (\alpha _{\mu ,\nu ,N} )|}{\sin ^N \zeta } \\ 
& +\left|\frac{\cos (\pi \mu )}{\cos (\pi \Re\mu )}\right|\left| a_{N+1} (\Re\mu )\right| 
\Gamma \left( \Re\nu  - N - \tfrac{1}{2} \right) \frac{|\cos (\alpha _{\mu ,\nu ,N} )|}{\sin ^{N+1} \zeta } .
\end{split}
\end{gather}
If $2\Re\mu $ is an odd integer, then the limiting values have to be taken in these bounds.
\end{theorem}

In the following two theorems, we give computable error bounds for the factorial expansions \eqref{formal10} and \eqref{formal11} of the Ferrers functions. The results in Theorem \ref{thm231} are especially useful when $\Re \nu \to +\infty$ and $\Im \nu$ is bounded, whereas those in Theorem \ref{thm23} are useful when $|\nu| \to +\infty$ in the larger domain $|\arg \nu|\leq\pi-\delta$ ($<\pi$).

\begin{theorem}\label{thm231} Let $0<\zeta<\pi$ and $N$ be an arbitrary non-negative integer. Let $\mu$ and $\nu$ be complex numbers satisfying $- N - \frac{1}{2} < \Re \mu  < \frac{1}{2}$ and $\Re \nu> -\Re \mu-1$. Then
\begin{align}
&\mathsf{P}_\nu ^\mu  (\cos\zeta)  = \sqrt{\frac{2}{\pi\sin\zeta}}\Gamma(\nu+\mu+1)
\left(\sum_{n=0}^{N-1} \frac{a_n(\mu)\cos(\beta_{\mu,\nu,n})}{\Gamma\left(\nu+\frac32+n\right) \sin^n\zeta} +\widehat{R}_N^{(\mathsf{P})}(\zeta,\mu,\nu)\right),\label{eqfp} \\
&\mathsf{Q}_\nu ^\mu  (\cos\zeta)  = -\sqrt{\frac{\pi}{2\sin\zeta}}\Gamma(\nu+\mu+1)
\left(\sum_{n=0}^{N-1} \frac{a_n(\mu)\sin(\beta_{\mu,\nu,n})}{\Gamma\left(\nu+\frac32+n\right) \sin^n\zeta} +\widehat{R}_N^{(\mathsf{Q})}(\zeta,\mu,\nu)\right),\label{eqfq}
\end{align}
where $\beta_{\mu,\nu,n}$ is defined in \eqref{phasedef1}, and the remainder terms satisfy the estimates
\begin{gather}\label{boundpq}
\begin{split}
\left|\widehat{R}_N^{(\mathsf{P})}(\zeta,\mu,\nu)\right|,& \left|\widehat{R}_N^{(\mathsf{Q})}(\zeta,\mu,\nu)\right|\leq \left| \frac{\cos (\pi \mu )}{\cos (\pi \Re \mu )} \right|\left| \frac{\Gamma \left( \frac{1}{2} - \Re \mu  \right)\Gamma (\Re \nu  + \Re \mu  + 1)}{\Gamma \left( \frac{1}{2} - \mu  \right)\Gamma (\nu  + \mu  + 1)}\right| \\ & \times  \frac{\left| a_N (\Re \mu ) \right|\cosh (\Im (\beta _{\mu ,\nu ,N} ))}{\Gamma \left( \Re \nu + \frac{3}{2} + N  \right)\sin ^N \zeta } \times \begin{cases}  \left| \sec \zeta  \right| & \text{ if } \;
0 < \zeta  \le \frac{\pi}{4}\; \text{ or } \; \frac{3\pi}{4} \le \zeta  < \pi, \\ 2\sin \zeta  & \text{ if } \; \frac{\pi}{4} < \zeta  < \frac{3\pi}{4},\end{cases}
\end{split}
\end{gather}
\begin{gather}\label{boundp}
\begin{split}
\left|\widehat{R}_N^{(\mathsf{P})}(\zeta,\mu,\nu)\right| \leq \; & \left| \frac{\cos (\pi \mu )}{\cos (\pi \Re \mu )} \right|\left| \frac{\Gamma \left( \frac{1}{2} - \Re \mu  \right)\Gamma (\Re \nu  + \Re \mu  + 1)}{\Gamma \left( \frac{1}{2} - \mu  \right)\Gamma (\nu  + \mu  + 1)}\right| 
\left(\frac{\left| a_N (\Re \mu ) \cos  ( \beta _{\mu ,\nu ,N} )\right|}{\Gamma \left( \Re \nu + \frac{3}{2} + N  \right)\sin ^N \zeta } \right.\\ & \left.  + \frac{\left| a_{N+1} (\Re \mu ) \sin  ( \beta _{\mu ,\nu ,N} )\right|}{\Gamma \left( \Re \nu + \frac{5}{2} + N  \right)\sin ^{N+1} \zeta }\right)
 \times \begin{cases}  \left| \sec \zeta  \right| & \text{ if } \;
0 < \zeta  \le \frac{\pi}{4}\; \text{ or } \; \frac{3\pi}{4} \le \zeta  < \pi, \\ 2\sin \zeta  & \text{ if } \; \frac{\pi}{4} < \zeta  < \frac{3\pi}{4},\end{cases}
\end{split}
\end{gather}
and
\begin{gather}\label{boundq}
\begin{split}
\left|\widehat{R}_N^{(\mathsf{Q})}(\zeta,\mu,\nu)\right| \leq \; & \left| \frac{\cos (\pi \mu )}{\cos (\pi \Re \mu )} \right|\left| \frac{\Gamma \left( \frac{1}{2} - \Re \mu  \right)\Gamma (\Re \nu  + \Re \mu  + 1)}{\Gamma \left( \frac{1}{2} - \mu  \right)\Gamma (\nu  + \mu  + 1)}\right| 
\left(\frac{\left| a_N (\Re \mu ) \sin  ( \beta _{\mu ,\nu ,N} )\right|}{\Gamma \left( \Re \nu + \frac{3}{2} + N  \right)\sin ^N \zeta } \right.\\ & \left.  + \frac{\left| a_{N+1} (\Re \mu ) \cos  ( \beta _{\mu ,\nu ,N} )\right|}{\Gamma \left( \Re \nu + \frac{5}{2} + N  \right)\sin ^{N+1} \zeta }\right)
 \times \begin{cases}  \left| \sec \zeta  \right| & \text{ if } \;
0 < \zeta  \le \frac{\pi}{4}\; \text{ or } \; \frac{3\pi}{4} \le \zeta  < \pi, \\ 2\sin \zeta  & \text{ if } \; \frac{\pi}{4} < \zeta  < \frac{3\pi}{4}.\end{cases}
\end{split}
\end{gather}
If $2\Re\mu $ is an odd integer, then the limiting values have to be taken in these bounds.
\end{theorem}

\begin{theorem}\label{thm23} Let $0<\zeta<\pi$ and $N$ be an arbitrary non-negative integer. Let $\mu$ and $\nu$ be complex numbers satisfying $-\frac12<\Re\mu<N+\frac12$ and $|\arg(\nu+\mu+1)|<\pi$. Then the expansions \eqref{eqfp} and \eqref{eqfq} hold with the remainder estimates
\begin{equation}\label{boundpq2}
\left|\widehat{R}_N^{(\mathsf{P})}(\zeta,\mu,\nu)\right|, \left|\widehat{R}_N^{(\mathsf{Q})}(\zeta,\mu,\nu)\right|\leq B_N \left|\frac{a_N(\mu)\cosh(\Im(\beta_{\mu,\nu,N}))}{\Gamma(\nu+\frac32+N)\sin^N\zeta}\right|,
\end{equation}
\begin{equation}\label{boundp2}
\left|\widehat{R}_N^{(\mathsf{P})}(\zeta,\mu,\nu)\right| \le B_N \left| \frac{ a_N (\mu )\cos (\beta _{\mu ,\nu ,N} )}{\Gamma \left( \nu  + \frac{3}{2} + N \right)\sin ^N \zeta } \right| + B_{N + 1} \left| \frac{a_{N + 1} (\mu )\sin (\beta _{\mu ,\nu ,N} )}{\Gamma \left( \nu  + \frac{5}{2} + N \right)\sin ^{N + 1} \zeta } \right|
\end{equation}
and
\begin{equation}\label{boundq2}
\left|\widehat{R}_N^{(\mathsf{Q})}(\zeta,\mu,\nu)\right| \le B_N \left| \frac{ a_N (\mu )\sin (\beta _{\mu ,\nu ,N} )}{\Gamma \left( \nu  + \frac{3}{2} + N \right)\sin ^N \zeta } \right| + B_{N + 1} \left| \frac{a_{N + 1} (\mu )\cos (\beta _{\mu ,\nu ,N} )}{\Gamma \left( \nu  + \frac{5}{2} + N \right)\sin ^{N + 1} \zeta } \right|.
\end{equation}
Here
\[
B_N =B_N(\mu,\nu,\sigma) =\left|\frac{N}{\frac12+\mu}\right|
+\left|\frac{\Gamma\left(N+\frac12-\Re\mu\right)\Gamma\left(\nu+\frac32+N\right)}{\Gamma\left(N+\frac12-\mu\right)\Gamma(\nu+\mu+1)}
\frac{\e^{(\pi+|\sigma|)\left|\Im\mu\right|}\left(1+\frac{N}{\frac12+\mu}\right)\left(\frac12\cos\sigma\right)^{-\frac12-\Re\mu}}{%
\left((\nu+\mu+1)\cos(\theta+\sigma)\right)^{N+\frac12-\Re\mu}}
\right|,
\]
in which $\theta=\arg(\nu+\mu+1)$, $|\theta+\sigma|<\frac{\pi}{2}$, $|\sigma|<\frac{\pi}{2}$. The fractional powers are taking their principal values. In the case that $|\theta|<\frac{\pi}{2}$ we take $\theta+\sigma=0$, that is, $\cos(\theta+\sigma)=1$. Note that $B_N=\O(1)$ as $|\nu|\to+\infty$.
\end{theorem}

\section{Gegenbauer function for large \texorpdfstring{$\nu$}{nu} and fixed \texorpdfstring{$\lambda$}{lambda}: error bounds}\label{GegenbauerBounds}

As a corollary of Theorem \ref{thm1} and the relation \eqref{eq69}, we provide error bounds for the inverse factorial expansion \eqref{formal14} of the Gegenbauer function as given in the following theorem.

\begin{corollary}\label{gegen1} Let $\xi \in \mathcal{D}_1$ and $N$, $M$ be arbitrary non-negative integers. Let $\lambda$ and $\nu$ be complex numbers satisfying $\left|\Re\lambda - \frac{1}{2}\right| <\min \left( N + \frac{1}{2},M + \frac{1}{2} \right)$ and $\Re\nu  > \max \left( N - \Re\lambda ,M - \Re\lambda \right)$. Then
\begin{gather}\label{Cexpansion}
\begin{split}
& C_\nu ^{(\lambda )} (\cosh \xi ) = \frac{\e^{\xi (\nu  + \lambda )} }{\Gamma (\lambda )\left(2\sinh \xi \right)^\lambda  }\left( \sum\limits_{n = 0}^{N - 1} a_n \!
\left( \lambda  - \tfrac{1}{2} \right)\frac{\Gamma (\nu  + \lambda -n)}{\Gamma (\nu+1)}\left( \frac{-\e^{ - \xi } }{\sinh \xi } \right)^n+ R_N^{(C1)} (\xi ,\lambda ,\nu ) \right)
\\  & + K(\xi ,\lambda )\frac{\e^{ - \xi (\nu  + \lambda )} }{\Gamma (\lambda )\left(2\sinh \xi \right)^\lambda  }\left( \sum\limits_{m = 0}^{M - 1} a_m\!
\left( \lambda  - \tfrac{1}{2} \right)\frac{\Gamma (\nu  + \lambda  - m)}{\Gamma (\nu  + 1)}\left( \frac{\e^\xi }{\sinh \xi } \right)^m   
+ R_M^{(C2)} (\xi ,\lambda ,\nu ) \right),
\end{split}
\end{gather}
where $K(\xi ,\lambda )$ is defined in \eqref{phasedef2}, and the remainder terms satisfy the estimates
\begin{align*}
\left| R_N^{(C1)} (\xi ,\lambda ,\nu ) \right| \le \; & \left|\frac{\sin (\pi \lambda )}{\sin (\pi \Re\lambda )}\right|\left| a_N \!\left(\Re\lambda  - \tfrac{1}{2}  \right) \right|  
\frac{\Gamma \left( \Re\nu+\Re\lambda - N \right)}{|\Gamma (\nu  + 1)|}   \left| \frac{\e^{ - \xi } }{\sinh \xi } \right|^N \\ 
& \times \begin{cases} 1 & \text{if }\; \Re (\e^{2\xi } ) \le 1, \\ \min\big(\left| 1 - \e^{ - 2\xi } \right|\left| \csc (2\Im\xi ) \right|,1 + \chi\big(N+\tfrac{1}{2}\big)\big) 
& \text{if } \; \Re (\e^{2\xi } ) > 1, \end{cases}
\end{align*}
and
\[
\left| R_M^{(C2)} (\xi ,\lambda ,\nu ) \right| \le \left|\frac{\sin (\pi \lambda )}{\sin (\pi \Re\lambda )}\right|\left|a_M\! \left(\Re\lambda   - \tfrac{1}{2} \right)\right|\frac{\Gamma \left( \Re\nu+\Re\lambda - M \right)}{|\Gamma (\nu  + 1)|} \left| \frac{\e^\xi}{\sinh \xi } \right|^M .
\]
If $\Re\lambda $ is an integer, then the limiting values have to be taken in these bounds. The fractional powers are defined to be positive for positive real $\xi$ and are defined by continuity elsewhere. In addition, the remainder term $R_M^{(C2)} (\xi ,\lambda,\nu)$ does not exceed the corresponding first neglected term in absolute value and has the same sign provided that $\xi$ is positive, $\lambda$ and $\nu$ are real, and $\left| \lambda - \frac{1}{2}\right| <M + \frac{1}{2}$ and $\nu > M - \lambda$.
\end{corollary}

Note that if $\nu$ is a positive integer, $C_\nu ^{(\lambda )}(\cosh\xi)$ is a polynomial in $\cosh\xi$. Thus, we can extend this result to the larger set $\widetilde {\mathcal{D}_1} =\left\{\xi : \Re\xi >0, -\pi < \Im \xi \leq\pi \right\}$. The function $\cosh\xi$ is a continuous bijection between $\widetilde {\mathcal{D}_1}$ and $\mathbb{C}\setminus \left[-1,1\right]$.

An immediate corollary of Theorem \ref{thm3} and the relation \eqref{eq66} are the following error bounds for the inverse factorial expansion \eqref{formal15} of the Gegenbauer function.

\begin{corollary}\label{gegen2} Let $0<\zeta<\pi$ and $N$ be an arbitrary non-negative integer. Let $\lambda$ and $\nu$ be complex numbers satisfying 
$\left| \Re\lambda  - \frac{1}{2} \right| < N + \frac{1}{2}$ and $\Re\nu  > N - \Re\lambda$. Then
\begin{equation}\label{Cexp2}
C_\nu ^{(\lambda )} (\cos \zeta ) = \frac{2}{\Gamma (\lambda )\left(2\sin\zeta\right)^\lambda }\left( \sum\limits_{n = 0}^{N - 1}a_n \!\left( \lambda  - \tfrac{1}{2} \right)
\frac{\Gamma (\nu  + \lambda  - n)}{\Gamma (\nu  + 1)}\frac{\cos (\gamma _{\lambda,\nu, n} )}{\sin ^n \zeta }  + R_N^{(C)} (\zeta ,\lambda ,\nu ) \right),
\end{equation}
where $\gamma_{\lambda,\nu,n}$ is defined in \eqref{phasedef2}, and the remainder term satisfies the estimate
\[
\left| R_N^{(C)} (\zeta ,\lambda ,\nu ) \right| \le \left|\frac{\sin (\pi \lambda )}{\sin (\pi \Re\lambda )}\right| \left| a_N \!\left( \Re\lambda  - \tfrac{1}{2} \right) \right|
\frac{\Gamma (\Re\nu+\Re\lambda  - N)}{\left| {\Gamma (\nu  + 1)} \right|}\frac{\cosh (\Im (\gamma _{\lambda ,\nu ,N} ))}{\sin ^N \zeta}.
\]
With the modified conditions $\left| \Re\lambda  - \frac{1}{2} \right| < N + \frac{1}{2}$ and $\Re\nu  > N - \Re\lambda +1$, we also have
\begin{align*}
\left| R_N^{(C)} (\zeta ,\lambda ,\nu ) \right| \le \; & \left|\frac{\sin (\pi \lambda )}{\sin (\pi \Re\lambda )}\right|\left| a_N \!\left( \Re\lambda  - \tfrac{1}{2} \right) \right|
\frac{\Gamma (\Re\nu+\Re\lambda  - N)}{\left| \Gamma (\nu  + 1) \right|}\frac{\left| \cos (\gamma _{\lambda ,\nu ,N} )\right|}{\sin ^N \zeta } \\ 
& + \left|\frac{\sin (\pi \lambda )}{\sin (\pi \Re\lambda )}\right|\left| a_{N + 1} \!\left( \Re\lambda  - \tfrac{1}{2} \right) \right|
\frac{\Gamma (\Re\nu+\Re\lambda - N - 1)}{\left| \Gamma (\nu  + 1) \right|}\frac{\left| \sin (\gamma _{\lambda ,\nu ,N} ) \right|}{\sin ^{N + 1} \zeta } .
\end{align*}
If $\Re\lambda $ is an integer, then the limiting values have to be taken in these bounds.
\end{corollary}

An immediate corollary of Theorems \ref{thm231}, \ref{thm23} and the relation \eqref{eq66} are the following error bounds for the factorial expansion \eqref{formal16} of the Gegenbauer function.

\begin{corollary}\label{gegen3} Let $0<\zeta<\pi$ and $N$ be an arbitrary non-negative integer. Let $\lambda$ and $\nu$ be complex numbers satisfying $0<\Re\lambda <N+1$ and $\Re \nu> -1$. Then
\begin{equation}\label{eq67}
C_\nu ^{(\lambda )} (\cos \zeta ) = \frac{2}{\Gamma (\lambda )(2\sin \zeta )^\lambda}\left( \sum\limits_{n = 0}^{N - 1} a_n\! \left( \lambda  - \tfrac{1}{2} \right)\frac{\Gamma (\nu  + 2\lambda )}{\Gamma \left( \nu  + \lambda  + n + 1 \right)}\frac{\cos (\delta _{\lambda ,\nu ,n} )}{\sin ^n \zeta }  + \widehat{R}_N^{(C)} (\zeta ,\lambda ,\nu ) \right),
\end{equation}
where $\delta_{\lambda,\nu,n}$ is defined in \eqref{phasedef2}, and the remainder term satisfies the estimates
\begin{gather}\label{eq68}
\begin{split}
\left|\widehat{R}_N^{(C)}(\zeta,\lambda,\nu)\right| \leq \; & \left| \frac{\sin (\pi \lambda )}{\sin (\pi \Re \lambda )} \right|\left| \frac{\Gamma (\Re \lambda )\Gamma (\Re \nu  + 1)}{\Gamma (\lambda )\Gamma (\nu  + 1)} \right|  \left| a_N \! \left( \Re \lambda  - \tfrac{1}{2} \right) \right|\frac{\left| \Gamma (\nu  + 2\lambda ) \right|}{\Gamma \left( \Re \nu  + \Re \lambda  + N + 1 \right)}\\ & \times \frac{\cosh (\Im (\delta _{\lambda ,\nu ,N} ))}{\sin ^N \zeta } \times \begin{cases}  \left| \sec \zeta  \right| & \text{ if } \;
0 < \zeta  \le \frac{\pi}{4}\; \text{ or } \; \frac{3\pi}{4} \le \zeta  < \pi, \\ 2\sin \zeta  & \text{ if } \; \frac{\pi}{4} < \zeta  < \frac{3\pi}{4}, \end{cases}
\end{split}
\end{gather}
and
\begin{align*}
\left|\widehat{R}_N^{(C)}(\zeta,\lambda,\nu)\right| \leq \; & \left| \frac{\sin (\pi \lambda )}{\sin (\pi \Re \lambda )} \right|\left| \frac{\Gamma (\Re \lambda )\Gamma (\Re \nu  + 1)}{\Gamma (\lambda )\Gamma (\nu  + 1)} \right| 
\left(\left| a_N \! \left( \Re \lambda  - \tfrac{1}{2} \right) \right|\frac{\left| \Gamma (\nu  + 2\lambda ) \right|}{\Gamma \left( \Re \nu  + \Re \lambda  + N + 1 \right)}\frac{|\cos (\delta _{\lambda ,\nu ,N} )|}{\sin ^N \zeta } \right.\\ & \left.  + \left| a_{N+1} \! \left( \Re \lambda  - \tfrac{1}{2} \right) \right|\frac{\left| \Gamma (\nu  + 2\lambda ) \right|}{\Gamma \left( \Re \nu  + \Re \lambda  + N + 2 \right)}\frac{|\sin (\delta _{\lambda ,\nu ,N} )|}{\sin ^{N+1} \zeta }\right)
\\ & \times \begin{cases}  \left| \sec \zeta  \right| & \text{ if } \;
0 < \zeta  \le \frac{\pi}{4}\; \text{ or } \; \frac{3\pi}{4} \le \zeta  < \pi, \\ 2\sin \zeta  & \text{ if } \; \frac{\pi}{4} < \zeta  < \frac{3\pi}{4}.\end{cases}
\end{align*}
If $\Re\lambda $ is an integer, then the limiting values have to be taken in these bounds.
\end{corollary}

\begin{corollary}\label{gegen4} Let $0<\zeta<\pi$ and $N$ be an arbitrary non-negative integer. Let $\lambda$ and $\nu$ be complex numbers satisfying $- N < \Re \lambda < 1$ and $|\arg(\nu+1)|<\pi$. Then the expansion \eqref{eq67} holds with the remainder estimates
\[
\left|\widehat{R}_N^{(C)}(\zeta,\lambda,\nu)\right| \leq K_N \left| a_N\! \left( \lambda  - \tfrac{1}{2} \right)\frac{\Gamma (\nu  + 2\lambda )}{\Gamma (\nu  + \lambda  + N + 1)}\frac{\cosh (\Im (\delta _{\lambda ,\nu ,N} ))}{\sin ^N \zeta } \right|
\]
and
\begin{align*}
\left|\widehat{R}_N^{(C)}(\zeta,\lambda,\nu)\right| \le\; & K_N \left| a_N\! \left( \lambda  - \tfrac{1}{2}\right)\frac{\Gamma (\nu  + 2\lambda )}{\Gamma (\nu  + \lambda  + N + 1)}\frac{\cos (\delta _{\lambda ,\nu ,N} )}{\sin ^N \zeta } \right| \\& + K_{N + 1} \left| a_{N + 1} \! \left( \lambda  - \tfrac{1}{2} \right)\frac{\Gamma (\nu  + 2\lambda )}{\Gamma (\nu  + \lambda  + N + 2)}\frac{\sin (\delta _{\lambda ,\nu ,N} )}{\sin ^{N + 1} \zeta } \right|.
\end{align*}
Here
\[
K_N =K_N(\lambda,\nu,\sigma) =\left| \frac{N}{1 - \lambda } \right| + \left| \frac{\Gamma (\Re \lambda +N)\Gamma (\nu  + \lambda  + N + 1)}{\Gamma (\lambda +N)\Gamma (\nu  + 1)}\frac{\e^{(\pi+|\sigma|) \left| \Im \lambda  \right|} \left( 1 + \frac{N}{1 - \lambda } \right)\left( \frac{1}{2}\cos \sigma  \right)^{\Re \lambda  - 1} }{((\nu  + 1)\cos (\theta  + \sigma ))^{\Re \lambda +N} } \right|,
\]
in which $\theta=\arg(\nu+1)$, $|\theta+\sigma|<\frac{\pi}{2}$, $|\sigma|<\frac{\pi}{2}$. The fractional powers are taking their principal values. In the case that $|\theta|<\frac{\pi}{2}$ we take $\theta+\sigma=0$, that is, $\cos(\theta+\sigma)=1$. Note that $K_N=\O(1)$ as $|\nu|\to+\infty$.
\end{corollary}

\section{Proof of Theorems \ref{thm22a} and \ref{thm22}}\label{SecHypRepresentation}

In this section, we prove the bounds for the large-$c$ expansion of the hypergeometric function given in Theorems \ref{thm22a}--\ref{thm22}. Throughout the section, we will assume that $c$ is not zero or a negative integer. We begin with the proof of the estimate \eqref{est1}. Starting with the standard integral representation \cite[\href{https://dlmf.nist.gov/15.6.E1}{Eq.\ 15.6.1}]{NIST:DLMF} and using Taylor's formula with integral remainder for $\left(1-zt\right)^{-a}$, we obtain the truncated expansion \eqref{hypergeomexp} with
\begin{gather}\label{remainder}
\begin{split}
R_N^{(F)}(a,b,c, z) = \frac{(a)_N (b)_N }{(c)_N N!}z^N & \frac{N\Gamma (c + N)}{\Gamma (b + N)\Gamma (c - b)} \\ & \times \int_0^1 t^{b + N - 1} (1 - t)^{c - b - 1} \int_0^1 \frac{(1 - u)^{N - 1} }{(1 - ztu)^{a + N} }\d u \d t,
\end{split}
\end{gather}
provided that $\Re c >\Re b$, $N > -\Re b$, $\left| \arg (-z) \right| < \pi$. With the added conditions $N > -\Re a$, $\left| \arg (-z) \right| \leq \frac{\pi }{2}$, we find
\begin{align*}
\left|\int_0^1 \frac{(1 - u)^{N - 1} }{(1 - ztu)^{a + N} }\d u \right|&\leq
\int_0^1 \frac{(1 - u)^{N - 1} }{\left| 1 - ztu \right|^{\Re  a  + N} }\e^{(\Im a)\arg (1 - ztu)} \d u\\
& \le\int_0^1 (1 - u)^{N - 1} \d u \mathop {\sup }\limits_{r > 0} \e^{(\Im a)\arg (1 + r\e^{\im \arg (-z)} )} 
\le\frac{\max (1,\e^{(\Im a)\arg (-z)} )}{N} .
\end{align*}
Thus
\begin{align*}
& \left|\frac{N\Gamma (c + N)}{\Gamma (b + N)\Gamma (c - b)} \int_0^1 t^{b + N - 1} (1 - t)^{c - b - 1} \int_0^1 \frac{(1 - u)^{N - 1} }{(1 - ztu)^{a + N} }\d u \d t \right|
\\ & \qquad\qquad\le \left|\frac{ \Gamma (c + N) }{ \Gamma (b + N)\Gamma (c - b) }\right|\int_0^1 t^{\Re  b  + N - 1} (1 - t)^{\Re  c - \Re b  - 1} \d t 
\max (1,\e^{(\Im a)\arg (-z)} )
\\ & \qquad\qquad= \left| \frac{\Gamma (\Re b + N)\Gamma (c + N)\Gamma (\Re c - \Re b)}{\Gamma (b + N)\Gamma (\Re c + N) \Gamma (c - b) }\right|\max (1,\e^{(\Im a)\arg (-z)} ) .
\end{align*}
In arriving at the penultimate expression, use has been made of the standard integral representation of the beta function \cite[\href{https://dlmf.nist.gov/5.12.E1}{Eq.\ 5.12.1}]{NIST:DLMF}. This completes the proof of \eqref{est1}.

Now suppose that $z<0$, $a$, $b$ and $c$ are real, and $N > \max(-a,-b)$ and $c>b$. From \eqref{remainder} and the mean value theorem of integration, we can assert that
\begin{align*}
R_N^{(F)} (a,b,c,z) =\; & \frac{(a)_N (b)_N }{(c)_N N!}z^N\frac{N\Gamma (c + N)}{\Gamma (b + N)\Gamma (c - b)} \\ & \times
\int_0^1 t^{b + N - 1} (1 - t)^{c - b - 1} \theta _N (t,z,a)\int_0^1 (1 - u)^{N - 1} \d u \d t
\\ =\; & \frac{(a)_N (b)_N }{(c)_N N!}z^N\frac{\Gamma (c + N)}{\Gamma (b + N)\Gamma (c - b)} \int_0^1 t^{b + N - 1} (1 - t)^{c - b - 1} \theta _N (t,z,a)\d t 
\\ =\; & \frac{(a)_N (b)_N }{(c)_N N!}z^N\Theta _N (z,a,b,c),
\end{align*}
where $0<\theta _N (t,z,a)<1$ and $0<\Theta _N (z,a,b,c)<1$. Thus, under these assumptions, the remainder term $R_N^{(F)} (a,b,c,z)$ does not
exceed the corresponding first neglected term in absolute value and has the same sign.

We continue with the proof of the inequality \eqref{est2}. An integral representation for the beta function \cite[\href{https://dlmf.nist.gov/5.12.E2}{Eq.\ 5.12.2}]{NIST:DLMF} gives
\begin{equation}\label{coeff}
 \frac{(a)_n }{n!} = \frac{2\sin (\pi a)}{\pi }\int_0^{\pi /2} \tan ^{2a - 1} u \sin ^{2n} u \d u,
\end{equation}
provided $n>-\Re a$ and $\Re a<1$. Hence, by summation and analytic continuation, we have
\[
(1 -zt)^{ - a}  = \sum\limits_{n = 0}^{N - 1} \frac{(a)_n }{n!}t^n z^n  + z^N t^N \frac{2\sin (\pi a)}{\pi }\int_0^{\pi /2} \frac{\tan ^{2a - 1} u\sin ^{2N} u}{1 -z t\sin ^2 u}\d u
\]
provided that $z\in \mathbb{C}\setminus \left[1,+\infty\right)$, $N>-\Re a$ and $\Re a<1$. Substitution into the standard integral representation \cite[\href{https://dlmf.nist.gov/15.6.E1}{Eq.\ 15.6.1}]{NIST:DLMF} followed by term-by-term integration, yields \eqref{hypergeomexp} with
\[
R_N^{(F)} (z,a,b,c) = \frac{2\sin (\pi a)}{\pi }\frac{\Gamma (c)}{\Gamma (b)\Gamma (c - b)}z^N \int_0^1 t^{b + N - 1} (1 - t)^{c - b - 1} \int_0^{\pi /2} \frac{\tan ^{2a - 1} u\sin ^{2N} u}{1 - zt\sin ^2 u}\d u\d t,
\]
under the assumptions that $z\in \mathbb{C}\setminus \left[1,+\infty\right)$, $N > \max ( - \Re a, - \Re b)$, $\Re a<1$ and $\Re c > \Re b$. 
Note that for $0<s<1$, 
\[
\left| 1 - zs \right|^2  = \left| z \right|^2 s^2 - 2(\Re z)s + 1 \ge \begin{cases} 1  & \text{ if } \;
\Re z \le  0, \\     \frac{(\Im z)^2}{ \left| z \right|^2} & \text{ if } \;
0<\Re z \le  \left| z \right|^2, \\  \left|1-z\right|^2 & \text{ if } \;  \Re z > \left| z \right|^2.\end{cases}
\]
Consequently,
\begin{multline*}
\left| R_N^{(F)} (z,a,b,c) \right| \le \frac{ 2\left|  \sin (\pi a) \right|} {\pi }\left| \frac{ \Gamma (c)}{\Gamma (b)\Gamma (c - b)}z^N  \right| \int_0^1  t^{\Re b + N - 1} (1 - t)^{\Re c - \Re b - 1} \d t
\\ \times  \int_0^{\pi /2}  \tan ^{2\Re a - 1} u\sin ^{2N} u\d u    \times \begin{cases} 1  & \text{ if } \;
\Re z \le  0, \\ \left|\frac{z}{\Im z}\right| & \text{ if } \;
0<\Re z \le  \left| z \right|^2, \\  \frac{1}{|1-z|} & \text{ if } \;  \Re z > \left| z \right|^2.\end{cases}
\end{multline*}
The $t$-integral can be computed using the standard formula for the beta function \cite[\href{https://dlmf.nist.gov/5.12.E1}{Eq.\ 5.12.1}]{NIST:DLMF}, and the explicit form of the $u$-integral follows from \eqref{coeff}. This completes the proof of the bound \eqref{est2}.

We conclude this section with the proof of Theorem \ref{thm22}. We will focus on the closed half-plane $\Re z \leq \frac12$. Combining formula \eqref{remainder} with the identity
\[
R_N^{(F)} (a,b,c,z) = \frac{(a)_N (b)_N }{(c)_N N!}z^N  + R_{N + 1}^{(F)} (a,b,c,z)
\]
and the standard integral representation \cite[\href{https://dlmf.nist.gov/15.6.E1}{Eq.\ 15.6.1}]{NIST:DLMF}, we obtain the truncated expansion \eqref{hypergeomexp}, in which we present the remainder term as
\[
R_N^{(F)} (z,a,b,c)=\frac{\left(a\right)_N \left(b\right)_N}{\left(c\right)_N N!} z^N+
\frac{\Gamma(c) \left(a\right)_{N+1} z^{N+1}}{\Gamma(b)\Gamma(c-b)(N+1)!}\int_0^1\left(1-t\right)^{c-b-1}t^{b+N}
\hyperF{a+N+1}{1}{N+2}{zt}\d t,
\]
provided that $N >  - \Re b - 1$ and $\Re c > \Re b$. Using the identities
\begin{align*}
\hyperF{a+N+1}{1}{N+2}{zt}&=
\left(1-zt\right)^{-a}\hyperF{1-a}{N+1}{N+2}{zt}\\
&=\frac{N+1}{azt}\left(1-zt\right)^{-a}\left(\hyperF{-a}{N}{N+1}{zt}-\hyperF{-a}{N+1}{N+1}{zt}\right)\\
&=\frac{N+1}{azt}\left( \left(1-zt\right)^{-a}\hyperF{-a}{N}{N+1}{zt}-1\right),
\end{align*}
we decompose the remainder $R_N^{(F)} (z,a,b,c)$ as follows:
\begin{align*}
R_N^{(F)} (z,a,b,c)&=\frac{\left(a\right)_N \left(b\right)_N}{\left(c\right)_N N!} z^N-\frac{\left(a+1\right)_N \left(b\right)_N}{\left(c\right)_N N!} z^N+\widetilde{R}_N^{(F)} (z,a,b,c)\\
&=\frac{-N}{a}\frac{\left(a\right)_N \left(b\right)_N}{\left(c\right)_N N!} z^N+\widetilde{R}_N^{(F)} (z,a,b,c),
\end{align*}
with
\[
\widetilde{R}_N^{(F)} (z,a,b,c)=\frac{\Gamma(c) \left(a+1\right)_{N} }{\Gamma(b)\Gamma(c-b)N!}z^{N}\int_0^1\frac{\left(1-t\right)^{c-b-1}t^{b+N-1}}{\left(1-zt\right)^{a}}
\hyperF{-a}{N}{N+1}{zt}\d t.
\]
The integral on the right-hand side is convergent provided that $N >  - \Re b$ and $\Re c > \Re b$. Now, by employing the substitution $t=1-\e^{-\tau}$, we deduce
\[
\widetilde{R}_N^{(F)} (z,a,b,c)=\frac{\Gamma(c) \left(a+1\right)_N }{\Gamma(b)\Gamma(c-b)N!}z^{N}\int_0^{+\infty}
\frac{\e^{(b-c)\tau}\left(1-\e^{-\tau}\right)^{b+N-1}}{\left(1-z\left(1-\e^{-\tau}\right)\right)^a}
\hyperF{-a}{N}{N+1}{z\left(1-\e^{-\tau}\right)}\d \tau.
\]
With this integral representation one can show that $R_N^{(F)} (z,a,b,c)=\O\left( |c|^{-N}\right)$ as $|c|\to+\infty$ in the sector $\left |\arg c\right|<\frac{\pi}{2}$. To extend this result to the larger sector $\left |\arg c\right|<\pi$, we deform the contour of integration by rotating it through an angle $\sigma$, $|\sigma|\leq\frac{\pi}{2}$. Thus, by writing $\tau=t\e^{\im\sigma}$, we arrive at
\[
\widetilde{R}_N^{(F)} (z,a,b,c)=\frac{\Gamma(c) \left(a+1\right)_N }{\Gamma(b)\Gamma(c-b)N!}z^{N}\int_0^{+\infty}\e^{(b-c)t\e^{\im\sigma}}t^{b+N-1}f(t)\d t,
\]
with
\[
f(t)=\e^{\im\sigma}\left(\frac{1-\e^{-t\e^{\im\sigma}}}{t}\right)^{b+N-1}\left(1-z\left(1-\e^{-t\e^{\im\sigma}}\right)\right)^{-a}
\hyperF{-a}{N}{N+1}{z\left(1-\e^{-t\e^{\im\sigma}}\right)}.
\]
Let $\theta=\arg(c-b)\in(-\pi,\pi)$, $|\theta+\sigma|<\frac{\pi}{2}$. Since $|\sigma|\leq\frac{\pi}{2}$, we have
\[
\left|\frac{1-\e^{-t\e^{\im\sigma}}}{t}\right|=\left|\frac{\e^{\im\sigma}}{t}\int_0^t\e^{-\tau\e^{\im\sigma}}\d\tau\right|\leq 1,
\]
and hence
\begin{align*}
\left|\left(\frac{1-\e^{-t\e^{\im\sigma}}}{t}\right)^{b+N-1}\right| & \le \exp \left(  - (\Im b)\arg \left(1 - \e^{ - t\e^{\im\sigma } } \right) \right) \\ & \le \exp \left( \left| \Im b \right|\left| \arg \left(1 - \e^{ - t\e^{\im\sigma } } \right) \right| \right) 
 \le \exp \left( \left| \Im b \right|\left| \sigma  \right| \right).
\end{align*}
The last inequality may be proved as follows. Assume that $|\sigma|<\frac{\pi}{2}$ and $t>0$. Then
\begin{align*}
& \left| \arg \left( 1 - \e^{ - t\e^{\im\sigma } } \right) \right|  = \arctan \left| \frac{\sin (t\sin \left| \sigma  \right|)}{\e^{t\cos \left| \sigma  \right|}  - \cos (t\sin \left| \sigma  \right|)} \right| \le \arctan \left| \frac{\sin (t\sin \left| \sigma  \right|)}{\e^{t\cos \left| \sigma  \right|}  - 1} \right|
\\ & \le \arctan \left( \frac{t\sin \left| \sigma  \right|}{\e^{t\cos \left| \sigma  \right|}  - 1} \right) = \arctan \left( \frac{t\cos \left| \sigma  \right|}{\e^{t\cos \left| \sigma  \right|}  - 1}\tan \left| \sigma  \right| \right) \le \arctan (\tan \left| \sigma  \right|) = \left| \sigma  \right|.
\end{align*}
The case $|\sigma|=\frac{\pi}{2}$ follows by continuity. For the hypergeometric function, we use the integral representation
\begin{equation}\label{eq111}
\left(1-z\left(1-\e^{-t\e^{\im\sigma}}\right)\right)^{-a}
\hyperF{-a}{N}{N+1}{z\left(1-\e^{-t\e^{\im\sigma}}\right)}
=N\int_0^1 \tau^{N-1}\left(\frac{1-\tau z\left(1-\e^{-t\e^{\im\sigma}}\right)}{1- z\left(1-\e^{-t\e^{\im\sigma}}\right)}\right)^{a}\d \tau.
\end{equation}
Since $|\sigma|\leq\frac{\pi}2$, we have
\begin{align*}
\left| \left( \frac{1 - \tau z\left( 1 - \e^{ - t\e^{\im\sigma } } \right)}{1 - z\left( 1 - \e^{ - t\e^{\im\sigma } }  \right)} \right)^a  \right| 
& = \left| \frac{1 - \tau z\left( 1 - \e^{ - t\e^{\im\sigma } }  \right)}{1 - z\left( 1 - \e^{ - t\e^{\im\sigma } } \right)} \right|^{\Re a} 
\exp \left(  - (\Im a)\arg \left( \frac{1 - \tau z\left( 1 - \e^{ - t\e^{\im\sigma } } \right)}{1 - z\left( 1 - \e^{ - t\e^{\im\sigma } } \right)} \right) \right) \\ 
& \le \left| \frac{1 - \tau z\left( 1 - \e^{ - t\e^{\im\sigma } } \right)}{1 - z\left( 1 - \e^{ - t\e^{\im\sigma } } \right)} \right|^{\Re a} \e^{\left| \Im a \right|\pi } . 
\end{align*}
We continue bounding the last expression under the assumption that $\Re a>0$. For fixed complex $w$ and $0\leq \tau\leq 1$, the function $\tau\mapsto\left|1-\tau w\right|$ has only one critical point, a local minimum. Hence, $\max_{0\leq \tau\leq 1}\left|1-\tau w\right|=\max(1,\left|1-w\right|)$. Consequently,
\[
\left|\frac{1-\tau z\left(1-\e^{-t\e^{\im\sigma}}\right)}{1- z\left(1-\e^{-t\e^{\im\sigma}}\right)}\right|\leq
\max\left(1,\left| 1- z\left(1-\e^{-t\e^{\im\sigma}}\right)\right|^{-1}\right).
\]
Since $t\geq0$ and $|\sigma|\leq\frac{\pi}{2}$, we have for $w=1-\e^{-t\e^{\im\sigma}}$ that $|w-1|\leq1$.
Taking  $\Re z\leq\frac12$, it follows that $1-zw=0$ can only happen in the two extreme cases, $\Re z=\frac12$ and $|\sigma|=\frac{\pi}{2}$. We consider these two cases separately.

(1) In the case that $\Re z<\frac12$, it follows from the maximum modulus principle that the minimum of $|1-zw|$ occurs when $|\sigma|=\frac{\pi}{2}$.
We consider the case $\sigma=\frac{\pi}{2}$ and take
\[
G(x,y,t) =\left|1-(x+\im y)(1-\e^{-\im t})\right|^2=2(1-\cos t)\left(\left(\tfrac12-x\right)^2+\left(y+\frac{\sin t}{2(1-\cos t)}\right)^2\right).
\]
The minimum occurs at $y=\frac{-\sin t}{2(1-\cos t)}$, that is, $1-\cos t=\frac{2}{4y^2+1}$.
Hence,
\[
G(x,y,t)\geq \frac{\left(\frac12-x\right)^2}{y^2+\frac14}.
\]
Thus
\[
\left| 1- z\left(1-\e^{-t\e^{\im\sigma}}\right)\right|^{-1}\leq \frac{\sqrt{\left(\Im z\right)^2+\frac14}}{\frac12-\Re z},
\]
which holds for $|\sigma|\leq \frac{\pi}{2}$ and tends to infinity when $z$ approaches the boundary $\Re z=\frac12$.

(2) In the case that $|\sigma|<\frac{\pi}{2}$, it follows from maximum modulus principle that the minimum of $|1-zw|$ occurs when $\Re z=\frac12$.
We take $z=\frac12+\im y$ and
\[
H(y,t,\sigma) =\left|1-(\tfrac12+\im y)(1-\e^{-t\e^{\im\sigma}})\right|^2.
\]
We expand the right-hand side, obtain a quadratic expression in $y$ and, hence, the minimum 
\[
H(y,t,\sigma)\geq\frac{\left(1-\e^{-2t\cos\sigma}\right)^2}{4\left(1-2\cos(t\sin\sigma)\e^{-t\cos\sigma}+\e^{-2t\cos\sigma}\right)}.
\]
Note that $\left(1-\e^{-2t\cos\sigma}\right)^2=\left(1-\e^{-t\cos\sigma}\right)^2\left(1+\e^{-t\cos\sigma}\right)^2\geq \left(1-\e^{-t\cos\sigma}\right)^2$.
Thus
\[
4H(y,t,\sigma)\geq\frac{\e^{t\cos\sigma}-2+\e^{-t\cos\sigma}}{\e^{t\cos\sigma}-2\cos(t\sin\sigma)+\e^{-t\cos\sigma}}\geq\cos^2\sigma,
\]
where the second inequality can be proved as follows. The inequality is equivalent to
\[
\cos^2\sigma\cos(t\sin\sigma)-1+\sin^2\sigma\cosh(t\cos\sigma)\geq0.
\]
Using $1=\cos^2\sigma+\sin^2\sigma$ we can re-write this inequality as
\[
\sin^2\sigma\sinh^2\left(\tfrac12 t\cos\sigma\right)\geq \cos^2\sigma\sin^2\left(\tfrac12 t\sin\sigma\right),
\]
which holds since for all positive $A$ and $B$ we have $\frac{\sinh^2A}{A^2}\geq 1\geq \frac{\sin^2 B}{B^2}$. Thus $4H(y,t,\sigma)\geq\cos^2\sigma$, that is,
\[
\left| 1- z\left(1-\e^{-t\e^{\im\sigma}}\right)\right|^{-1}\leq\frac{2}{\cos\sigma},
\]
which holds for $\Re z\leq\frac12$ and tends to infinity when $\sigma$ approaches the boundary $|\sigma|= \frac{\pi}{2}$.

Therefore, in summary, the absolute value of the left-hand side of \eqref{eq111} is at most $K^{\Re a} \e^{|\Im a |\pi}$, in which we can take
\[
K=\max\left(1,\frac{\sqrt{\left(\Im z\right)^2+\frac14}}{\frac12-\Re z}\right)\qquad\textrm{or}\qquad
K=\frac{2}{\cos\sigma}.
\]
Consequently,
\begin{align*}
\left|\widetilde{R}_N^{(F)} (z,a,b,c)\right| & \le \left| \frac{\Gamma (c)(a + 1)_N }{\Gamma (b)\Gamma (c - b)N!}z^N  \right|\int_0^{ + \infty } \e^{ - \left| c - b \right|t\cos (\theta  + \sigma )} t^{\Re b + N - 1} \left| f(t) \right|\d t 
\\ & \le \left| \frac{\Gamma (c)(a + 1)_N }{\Gamma (b)\Gamma (c - b)N!}z^N \right|\int_0^{ + \infty } \e^{ - \left| c - b \right|t\cos (\theta  + \sigma )} t^{\Re b + N - 1} \e^{\pi |\Im a| + |\sigma ||\Im b|} K^{\Re a} \d t
\\ & = \e^{\pi |\Im a| + |\sigma ||\Im b|} \left| \frac{\Gamma (\Re b+N)\Gamma (c)(a + 1)_N }{\Gamma (b)\Gamma (c - b)N!((c - b)\cos (\theta  + \sigma ))^{\Re b+N} }z^N  \right|K^{\Re a} ,
\end{align*}
and therefore,
\begin{align*}
& \left| R_N^{(F)} (z,a,b,c) \right| \le \left| \frac{N}{a}\frac{(a)_N (b)_N }{(c)_N N!}z^N \right| + \left| \widetilde{R}_N^{(F)} (z,a,b,c) \right|
\\ & \le \left| \frac{N}{a}\frac{(a)_N (b)_N }{(c)_N N!}z^N \right| + \e^{\pi |\Im a| + |\sigma ||\Im b|} \left| \frac{\Gamma (\Re b + N)\Gamma (c)(a + 1)_N }{\Gamma (b)\Gamma (c - b)N!((c - b)\cos (\theta  + \sigma ))^{\Re b + N} }z^N \right|K^{\Re a} 
\\ & = \left( \left| \frac{N}{a} \right| + \left| \frac{\Gamma (\Re b + N)\Gamma (c + N)}{\Gamma (b + N)\Gamma (c - b)}\frac{\e^{\pi |\Im a| + |\sigma ||\Im b|} \left( 1 + \frac{N}{a} \right)K^{\Re a} }{((c - b)\cos (\theta  + \sigma ))^{\Re b + N} } \right| \right)\left| \frac{(a)_N (b)_N }{(c)_N N!}z^N  \right| .
\end{align*}

\section{Proof of Theorems \ref{thm1} and \ref{thm2}}\label{SecProofs12}

In this section, we prove the error bounds stated in Theorems \ref{thm1}--\ref{thm2}. To this end, we require an estimate for the remainder term of the well-known asymptotic expansion of the modified Bessel function of the second kind given in the following lemma.

\begin{lemma}\label{lemma1} Let $N$ be a non-negative integer and let $\mu$ be an arbitrary complex number such that 
$\left|\Re\mu \right|<N + \frac{1}{2}$. Then
\begin{equation}\label{eq2}
K_{ \pm \mu } (w) = \sqrt{\frac{\pi }{2w}} \e^{ - w} \left(  \sum\limits_{n = 0}^{N - 1} \frac{a_n (\mu )}{w^n}  + r_N^{(K)} (w,\mu ) \right),
\end{equation}
where the remainder term satisfies the estimate
\[
\left| r_N^{(K)} (w,\mu ) \right| \le \left|\frac{\cos (\pi \mu )}{\cos (\pi \Re\mu )}\right|\frac{\left|a_N (\Re\mu )\right|}{\left| w \right|^N } \times 
\begin{cases} 1 & \text{if } \; |\arg w| \le \frac{\pi }{2}, \\ \min\big(\left|\csc (\arg w) \right|,1 + \chi\big(N+\frac{1}{2}\big)\big) 
& \text{if } \;  \frac{\pi }{2} < |\arg w| \leq \pi. \end{cases}
\]
If $2\Re\mu $ is an odd integer, then the limiting value has to be taken in this bound. In addition, the remainder term $r_N^{(K)} (w ,\mu)$ does not exceed 
the corresponding first neglected term in absolute value and has the same sign provided that $w$ is positive, $\mu$ is real, and $\left| \mu \right| <N + \frac{1}{2}$.
\end{lemma}

\begin{proof} The statement for complex variables follows from Theorem 1.8 and Propositions B.1 and B.3 of the paper \cite{Nemes17}. For the real case, see, e.g., \cite[pp.\ 206--207]{Watson1944}.
\end{proof}

We continue with the proof of Theorem \ref{thm1}. The associated Legendre function of the first kind can be represented in
terms of the modified Bessel function of the first kind as
\begin{equation}\label{eq4}
P_\nu ^\mu  (\cosh \xi ) = \frac{1}{\Gamma (\nu  - \mu  + 1)}\int_0^{ + \infty } t^\nu  \e^{ - t\cosh \xi } I_{ - \mu } (t\sinh \xi ) \d t,
\end{equation}
provided $\xi>0$ and $\Re\nu  > \Re\mu  - 1$ \cite[p.\ 214, Ent.\ 6.52]{Oberhettinger1974}. 
The growth rate of the modified Bessel function of the first kind at infinity (see, e.g., \cite[\href{http://dlmf.nist.gov/10.30.E4}{Eq.\ 10.30.4}]{NIST:DLMF}) 
and analytic continuation show that this representation is actually valid in the larger domain $\mathcal{D}_{\frac{1}{2}}$.

Now, suppose that $|\arg w|<2\pi$. Under this assumption, we have the following relations between the modified Bessel functions:
\[
I_{ - \mu } (w) =  \mp \frac{\im}{\pi }K_{ - \mu } (w\e^{ \mp \pi \im} ) \pm \frac{\im}{\pi }\e^{ \mp \pi \im\mu } K_{ - \mu } (w),
\]
when $0<\pm \arg w <2\pi$, and
\[
I_{ - \mu } (w) = \frac{\im}{2\pi }\left(K_{ - \mu } \left(w\e^{\pi \im} \right) - K_{ - \mu } \left(w\e^{ - \pi \im} \right)\right) + \frac{\sin (\pi \mu )}{\pi }K_{ - \mu } (w),
\]
when $\arg w = 0$ (cf. \cite[\href{https://dlmf.nist.gov/10.34.E3}{Eq.\ 10.34.3}]{NIST:DLMF}). The combination of these functional equations with the expression \eqref{eq2} yields
\begin{equation}\label{eq5}
I_{ - \mu } (w) = \frac{\e^w }{\sqrt{2\pi w}}\left( \sum_{n = 0}^{N - 1} ( - 1)^n \frac{a_n (\mu )}{w^n }  + \widetilde r_N^{(K)} (w,\mu ) \right) 
+ \widetilde C(\arg w,\mu )\frac{\e^{ - w} }{\sqrt{2\pi w}}\left( \sum_{m = 0}^{M - 1} \frac{a_m (\mu )}{w^m }  + r_M^{(K)} (w,\mu ) \right),
\end{equation}
with
\begin{equation}\label{eq7}
\widetilde C(\arg w,\mu )  = \begin{cases} \sin (\pi \mu ) & \text{if }\; \arg w = 0, \\ \pm \im \e^{ \mp \pi \im\mu } & \text{if } \; 0<\pm \arg w < 2\pi, \end{cases}
\end{equation}
and
\begin{equation}\label{eq11}
\widetilde r_N^{(K)} (w,\mu )  = \begin{cases} \frac{1}{2}(r_N^{(K)} (w\e^{\pi \im} ,\mu ) + r_N^{(K)} (w\e^{ - \pi \im} ,\mu )) & \text{if }\; \arg w = 0, \\ r_N^{(K)} (w\e^{ \mp \pi \im} ,\mu ) & \text{if } \; 0<\pm \arg w < 2\pi. \end{cases}
\end{equation}
Here $N$ and $M$ are arbitrary non-negative integers. Assuming that $\Re\nu  > \max \left( N - \frac{1}{2},M - \frac{1}{2} \right)$, the substitution of \eqref{eq5}
(with $w=t\sinh \xi$) into \eqref{eq4} gives \eqref{eq3} with $C(\xi,\mu ) = \widetilde C(\arg (\sinh \xi),\mu )$ and
\[
R_N^{(P1)} (\xi ,\mu ,\nu ) = \e^{ -  \left( \nu  + \frac{1}{2} \right)\xi}\int_0^{ + \infty } t^{\nu  - \frac{1}{2}} \e^{ - t\e^{ - \xi } } \widetilde r_N^{(K)} (t\sinh \xi ,\mu )\d t ,
\]
\[
R_M^{(P2)} (\xi ,\mu ,\nu ) =\e^{ \left( \nu  + \frac{1}{2} \right)\xi}\int_0^{ + \infty } t^{\nu  - \frac{1}{2}} \e^{ - t\e^\xi  } r_M^{(K)} (t\sinh \xi ,\mu )\d t .
\]
We perform a change of integration variable from $t$ to $s$ by $t=s \e^\xi$ in the case of the first integral and by $t=s \e^{-\xi}$ in the case of the second integral. 
By an application of Cauchy's theorem, the contours of integration can then be deformed back into the positive real axis and we arrive at the representations
\begin{equation}\label{eq12}
R_N^{(P1)} (\xi ,\mu ,\nu ) =\int_0^{ + \infty } s^{\nu  - \frac{1}{2}} \e^{ - s} \widetilde r_N^{(K)} \left(s \e^\xi  \sinh \xi ,\mu\right)\d s 
\end{equation}
and
\begin{equation}\label{eq13}
R_M^{(P2)} (\xi ,\mu ,\nu ) =\int_0^{ + \infty } s^{\nu  - \frac{1}{2}} \e^{ - s} r_M^{(K)} \left(s\e^{ - \xi } \sinh \xi ,\mu\right)\d s,
\end{equation}
respectively. Now observe that the function $\e^\xi \sinh\xi$ can be regarded as an analytic function from the domain $\mathcal{D}_1$ to the sector 
$\left\{z : |\arg z|<2\pi~\textrm{and}~ \left|z+\frac12\right|>\frac12\right\}$. It maps the part of $\mathcal{D}_1$ lying in the upper half-plane into 
the sector $\left\{z : 0<\arg z<2\pi~\textrm{and}~ \left|z+\frac12\right|>\frac12\right\}$, and it maps the part of $\mathcal{D}_1$ lying in the lower half-plane 
into the sector $\left\{z : -2\pi<\arg z<0~\textrm{and}~ \left|z+\frac12\right|>\frac12\right\}$. Similarly, the function $\e^{-\xi} \sinh\xi$ can be viewed as an 
analytic function from the domain $\mathcal{D}_1$ to the disk $\left\{z :  \left|z-\frac12\right|<\frac12\right\}$. Thus, we can employ analytic continuation 
in $\xi$ to extend our results into the domain $\mathcal{D}_1$. The formula \eqref{eq6} for $C(\xi,\mu )$ follows immediately from \eqref{eq7} and the 
properties of the hyperbolic sine function.

By Lemma \ref{lemma1} and the properties of the function $\e^{-\xi} \sinh\xi$, we can assert that
\begin{align*}
\left| R_M^{(P2)} (\xi ,\mu ,\nu ) \right| & \le\int_0^{ + \infty } s^{\Re\nu  - \frac{1}{2}} \e^{ - s} \left| r_M^{(K)} (s\e^{ - \xi } \sinh \xi ,\mu ) \right|\d s 
\\ & \le \left|\frac{\cos (\pi \mu )}{\cos (\pi \Re\mu )}\right|\left|a_M (\Re\mu )\right|\Gamma \left( \Re\nu  - M + \tfrac12 \right)  \left| \frac{\e^\xi}{\sinh \xi } \right|^M ,
\end{align*}
provided that $\xi \in \mathcal{D}_1$, $\left| \Re\mu  \right| < M + \frac{1}{2}$ and $\Re\nu  > M - \frac{1}{2}$. If $2\Re\mu $ is an odd integer, then the limiting 
value has to be taken in this bound. Similarly, by Lemma \ref{lemma1} and the properties of the function $\e^{\xi} \sinh\xi$, we can infer that
\begin{align*}
\left| R_N^{(P1)} (\xi ,\mu ,\nu ) \right| \le \; & \left|\frac{\cos (\pi \mu )}{\cos (\pi \Re\mu )}\right|\left| a_N (\Re\mu ) \right| 
\Gamma \left( \Re\nu  - N + \tfrac12 \right)  \left| \frac{\e^{ - \xi } }{\sinh \xi } \right|^N \\ 
& \times \begin{cases} 1 & \text{if } \; \frac{\pi }{2} \le \left| \vartheta \bmod 2\pi \right| \le \pi, \\ \min\big(\left|\csc \vartheta \right|,1 + \chi\big(N+\frac{1}{2}\big)\big) 
& \text{if } \; \left| \vartheta \bmod 2\pi \right| < \frac{\pi }{2}, \end{cases}
\end{align*}
with $\vartheta = \arg(\e^\xi  \sinh \xi) \in (-2\pi,2\pi)$, provided that $\xi \in \mathcal{D}_1$, $\left| \Re\mu  \right| < N + \frac{1}{2}$ and $\Re\nu  > N - \frac{1}{2}$. Again, if $2\Re\mu $ is an odd integer, then the limiting value has to be taken in this bound. Since
\[
\left|\tan\vartheta\right|=\left|\frac{\Im (\e^\xi  \sinh \xi )}{\Re (\e^\xi  \sinh \xi )}\right|
=\left| \frac{\e^{2\Re \xi } \sin (2\Im\xi )}{\e^{2\Re \xi } \cos (2\Im\xi ) - 1} \right|,
\]
whenever $\e^{2\Re \xi } \cos (2\Im\xi ) \ne 1$ (i.e., $\vartheta \not\equiv \frac{\pi}{2} \bmod \pi$), and
\[
\left| \csc \vartheta \right| = \frac{\sqrt {1 + \tan ^2 \vartheta } }{\left| \tan \vartheta \right|},
\]
it follows that $\left|\csc \vartheta \right| = \left| 1 - \e^{ - 2\xi } \right|\left| \csc (2\Im\xi ) \right|$. Finally, $\frac{\pi }{2} \le \left| \vartheta \bmod 2\pi \right| \le \pi$ occurs whenever $\Re (\e^\xi  \sinh \xi ) \le 0$, i.e., $\Re (\e^{2\xi } ) \le 1$, and $\left| \vartheta \bmod 2\pi \right| < \frac{\pi }{2}$ occurs whenever $\Re (\e^\xi  \sinh \xi ) > 0$, i.e., $\Re (\e^{2\xi } ) > 1$.

Now assume that $\xi>0$, $\mu$ and $\nu$ are real, and $\left| \mu \right| <M + \frac{1}{2}$ and $\nu > M - \frac{1}{2}$. 
From Lemma \ref{lemma1} and the mean value theorem of integration, we can infer that
\begin{align*}
R_M^{(P2)} (\xi ,\mu ,\nu ) & = \int_0^{ + \infty } s^{\nu  - \frac{1}{2}} \e^{ - s} \frac{a_M (\nu )}{(s \e^{ - \xi } \sinh \xi )^M }\theta _M (\xi ,\mu )\d s \\
& = a_M (\nu )\Gamma \left( \nu  - M + \tfrac12\right)\left( \frac{\e^\xi}{\sinh \xi } \right)^M \Theta _M (\xi ,\mu ,\nu ),
\end{align*}
where $0<\theta _M (\xi ,\mu )<1$ and $0<\Theta _M (\xi ,\mu ,\nu )<1$. This completes the proof of Theorem \ref{thm1}.

We conclude this section with the proof of Theorem \ref{thm2}.  Let $N$ be an arbitrary non-negative integer. 
Assuming that $\Re\nu  > N - \frac{1}{2}$, the substitution of \eqref{eq5} (with $w=t\sinh \xi$) into \eqref{QIntRepr1} gives \eqref{eq10} with
\[
R_N^{(Q)} (\xi ,\mu ,\nu ) =\e^{ \left( \nu  + \frac{1}{2} \right)\xi}\int_0^{ + \infty } t^{\nu  - \frac{1}{2}} \e^{ - t\e^\xi  } r_N^{(K)} (t\sinh \xi ,\mu )\d t.
\]
Hence, $R_N^{(Q)} (\xi ,\mu ,\nu ) = R_N^{(P2)} (\xi ,\mu ,\nu )$ and the estimates for $R_N^{(Q)} (\xi ,\mu ,\nu )$ follow from those for $R_N^{(P2)} (\xi ,\mu ,\nu )$.

\section{Proof of Theorem \ref{thm3}}\label{SecProof3}

In this section, we prove the error bounds stated in Theorem \ref{thm3}. The Ferrers function of the first kind is related to the associated Legendre function of the first kind via the limit
\begin{equation}\label{eq14}
\mathsf{P}_\nu ^\mu  (x) = \mathop {\lim }\limits_{\varepsilon  \to 0 + } \e^{\frac{\pi}{2} \im \mu} P_\nu ^\mu  (x + \im\varepsilon ),
\end{equation}
where $-1<x<1$ \cite[\href{http://dlmf.nist.gov/14.23.E1}{Eq.\ 14.23.1}]{NIST:DLMF}. Suppose that $0<\zeta<\pi$. Since
\begin{equation}\label{eq19}
\cosh (\varepsilon  + \im\zeta ) = \cosh \varepsilon \cos \zeta  + \im\sinh \varepsilon \sin \zeta,
\end{equation}
we obtain, using \eqref{eq14}, that
\begin{equation}\label{Plim}
\mathsf{P}_\nu ^\mu  (\cos \zeta ) = \mathop {\lim }\limits_{\varepsilon  \to 0 + } \e^{\frac{\pi}{2} \im\mu} P_\nu ^\mu  (\cosh (\varepsilon  + \im\zeta )).
\end{equation}
Assuming that  $\Re\nu  > \Re\mu  - 1$ and $\Re\nu  > N - \frac{1}{2}$, and taking into account \eqref{eq11}, \eqref{eq12} and \eqref{eq13}, 
the substitution of \eqref{eq3} (with $M=N$) into the right-hand side of \eqref{Plim} yields \eqref{eq15} with
\begin{gather}\label{eq18}
\begin{split}
R_N^{(\mathsf{P})} (\zeta ,\mu ,\nu ) =\; & \tfrac12 \e^{\left( \left( \nu  + \frac{1}{2} \right)\zeta  + \left( \mu  - \frac{1}{2} \right)\frac{\pi }{2} \right)\im} \mathop {\lim }\limits_{\varepsilon  \to 0 + } R_N^{(P1)} \left(\varepsilon+ \e^{\frac{\pi}{2}\im}\zeta,\mu ,\nu \right) \\ 
& + \tfrac12 \e^{ - \left( \left( \nu  + \frac{1}{2} \right)\zeta  + \left( \mu  - \frac{1}{2} \right)\frac{\pi }{2} \right)\im} \mathop {\lim }\limits_{\varepsilon  \to 0 + } R_N^{(P2)} \left(\varepsilon+\e^{\frac{\pi}{2}\im}\zeta,\mu ,\nu \right)
\\ = \; &\tfrac12 \e^{\left( \left( \nu  + \frac{1}{2} \right)\zeta  + \left( \mu  - \frac{1}{2} \right)\frac{\pi }{2} \right)\im} \int_0^{ + \infty } s^{\nu  - \frac{1}{2}} \e^{ - s} r_N^{(K)} (s\e^{\left( \zeta  - \frac{\pi }{2} \right)\im} \sin \zeta ,\mu )\d s
\\ & + \tfrac12 \e^{ - \left( \left( \nu  + \frac{1}{2} \right)\zeta  + \left( \mu  - \frac{1}{2} \right)\frac{\pi }{2} \right)\im} \int_0^{ + \infty } s^{\nu  - \frac{1}{2}} \e^{ - s} r_N^{(K)} (s\e^{ - \left( \zeta  - \frac{\pi }{2} \right)\im} \sin \zeta ,\mu )\d s.
\end{split}
\end{gather}
It follows from Lemma \ref{lemma1} that
\begin{gather}\label{eq22}
\begin{split}
\left| R_N^{(\mathsf{P})} (\zeta ,\mu ,\nu ) \right| \le \; & \tfrac12 \e^{-\Im\left( \alpha_{\mu,\nu,N}\right)} \int_0^{ + \infty } s^{\Re\nu   - \frac{1}{2}} \e^{ - s}
\left| r_N^{(K)} (s\e^{\left( \zeta  - \frac{\pi }{2} \right)\im} \sin \zeta ,\mu )\right|\d s\\
& + \tfrac12 \e^{\Im\left( \alpha_{\mu,\nu,N} \right)} \int_0^{ + \infty } s^{\Re\nu    - \frac{1}{2}} \e^{ - s} 
\left|r_N^{(K)} (s\e^{ - \left( \zeta  - \frac{\pi }{2} \right)\im} \sin \zeta ,\mu )\right|\d s \\ 
\le \; & \left|\frac{\cos (\pi \mu )}{\cos (\pi \Re\mu ) }\right|\left| a_N (\Re\mu )\right| \Gamma \left( \Re\nu  - N + \tfrac12 \right)
\frac{\cosh ( \Im(\alpha_{\mu,\nu,N}))}{\sin ^N \zeta } ,
\end{split}
\end{gather}
provided that $0<\zeta<\pi$, $\left| \Re\mu  \right| < N + \frac{1}{2}$ and $\Re\nu  > N - \frac{1}{2}$. If $2\Re\mu $ is an odd integer, then the limiting value 
has to be taken in this bound. This proves \eqref{eq16} for $R_N^{(\mathsf{P})} (\zeta ,\mu ,\nu )$. To prove \eqref{eq17}, we first re-write \eqref{eq18} in the form
\begin{align*}
R_N^{(\mathsf{P})} (\zeta ,\mu ,\nu ) = \; & \tfrac12 \cos (\alpha _{\mu ,\nu ,N} )\int_0^{ + \infty } s^{\nu  - \frac{1}{2}} \e^{ - s} \left( 
\e^{\im\zeta N} ( - \im)^N r_N^{(K)} (s\e^{\left( \zeta  - \frac{\pi }{2} \right)\im} \sin \zeta ,\mu ) \right.\\ 
& \hspace{124.5pt} \left.+ \e^{ - \im\zeta N} \im^N r_N^{(K)} (s\e^{ - \left( \zeta  - \frac{\pi }{2} \right)\im} \sin \zeta ,\mu ) \right)\d s
\\ & + \tfrac\im2 \sin (\alpha _{\mu ,\nu ,N} )\int_0^{ + \infty } s^{\nu  - \frac{1}{2}} \e^{ - s} \left( \e^{\im\zeta N} ( - \im)^N r_N^{(K)} (s\e^{\left( 
\zeta  - \frac{\pi }{2} \right)\im} \sin \zeta ,\mu ) \right.\\ 
& \hspace{124.5pt}\left.- \e^{ - \im\zeta N} \im^N r_N^{(K)} (s\e^{ - \left( \zeta  - \frac{\pi }{2} \right)\im} \sin \zeta ,\mu ) \right)\d s.
\end{align*}
Suppose that $\Re\nu  > N + \frac{1}{2}$. Next, we employ the relation $r_N^{(K)} (w,\mu ) = a_N (\mu )w^{ - N}  + r_{N + 1}^{(K)} (w,\mu )$ inside 
the second integral, to obtain
\begin{align*}
R_N^{(\mathsf{P})} (\zeta ,\mu ,\nu ) = \; & \tfrac12 \cos (\alpha _{\mu ,\nu ,N} )\int_0^{ + \infty } s^{\nu  - \frac{1}{2}} \e^{ - s} \left( 
\e^{\im\zeta N} ( - \im)^N r_N^{(K)} (s\e^{\left( \zeta  - \frac{\pi }{2} \right)\im} \sin \zeta ,\mu ) \right.\\ 
& \hspace{124.5pt} \left.+ \e^{ - \im\zeta N} \im^N r_N^{(K)} (s\e^{ - \left( \zeta  - \frac{\pi }{2} \right)\im} \sin \zeta ,\mu ) \right)\d s
\\ & + \tfrac\im2 \sin (\alpha _{\mu ,\nu ,N} )\int_0^{ + \infty } s^{\nu  - \frac{1}{2}} \e^{ - s} \left( \e^{\im\zeta N} ( - \im)^N r_{N+1}^{(K)} (s\e^{\left( 
\zeta  - \frac{\pi }{2} \right)\im} \sin \zeta ,\mu ) \right.\\ 
& \hspace{124.5pt}\left.- \e^{ - \im\zeta N} \im^N r_{N+1}^{(K)} (s\e^{ - \left( \zeta  - \frac{\pi }{2} \right)\im} \sin \zeta ,\mu ) \right)\d s.
\end{align*}
It follows from Lemma \ref{lemma1} that
\begin{gather}\label{eq25}
\begin{split}
\left| R_N^{(\mathsf{P})} (\zeta ,\mu ,\nu ) \right| \le \; &\left|\frac{\cos (\pi \mu )}{\cos (\pi \Re\mu ) }\right|\left| a_N (\Re\mu )\right| 
\Gamma \left( \Re\nu  - N + \tfrac12 \right) \frac{|\cos (\alpha _{\mu ,\nu ,N} )|}{\sin ^N \zeta } \\ & +
\left|\frac{\cos (\pi \mu )}{\cos (\pi \Re\mu ) }\right|\left| a_{N+1} (\Re\mu )\right| \Gamma \left( \Re\nu  - N - \tfrac12\right)
 \frac{|\sin (\alpha _{\mu ,\nu ,N} )|}{\sin ^{N+1} \zeta } ,
\end{split}
\end{gather}
provided that $0<\zeta<\pi$, $\left| \Re\mu  \right| < N + \frac{1}{2}$ and $\Re\nu  > N + \frac{1}{2}$. Again, if $2\Re\mu $ is an odd integer, 
then the limiting value has to be taken in this bound.

The Ferrers function of the second kind is related to the associated Legendre function of the second kind via the limit
\begin{equation}\label{eq59}
\mathsf{Q}_\nu ^\mu  (x) = \mathop {\lim }\limits_{\varepsilon  \to 0 + } \tfrac12\e^{ - \pi \im\mu }\left( \e^{ -\frac{\pi}{2} \im\mu} 
Q_\nu ^\mu  (x + \im\varepsilon ) + \e^{\frac{\pi}{2} \im\mu } Q_\nu ^\mu  (x - \im\varepsilon ) \right),
\end{equation}
where $-1<x<1$ (combine \cite[\href{http://dlmf.nist.gov/14.23.E5}{Eq.\ 14.23.5}]{NIST:DLMF} with 
\cite[\href{http://dlmf.nist.gov/14.3.E10}{Eq.\ 14.3.10}]{NIST:DLMF}). Suppose that $0<\zeta<\pi$. 
From \eqref{eq19} and \eqref{eq59}, we can assert that
\begin{equation}\label{eq20}
\mathsf{Q}_\nu ^\mu  (\cos \zeta ) = \mathop {\lim }\limits_{\varepsilon  \to 0 + } \tfrac12\e^{ - \pi \im\mu } \left( \e^{ - \frac{\pi}{2} \im\mu } 
Q_\nu ^\mu  (\cosh (\varepsilon  + \im\zeta )) + \e^{\frac{\pi}{2}\im\mu} Q_\nu ^\mu  (\cosh ( \varepsilon  - \im\zeta )) \right) .
\end{equation}
Assuming that  $\Re\nu  > \Re\mu  - 1$ and $\Re\nu  > N - \frac{1}{2}$, and taking into account \eqref{eq13}, the substitution of \eqref{eq10} into the right-hand side of \eqref{eq20} yields \eqref{eq21} with
\begin{gather}\label{eq24}
\begin{split}
R_N^{(\mathsf{Q})} (\zeta ,\mu ,\nu ) = & -\tfrac12\e^{- \left( \left( \nu  + \frac{1}{2} \right)\zeta  + \left( \mu  + \frac{1}{2} \right)\frac{\pi }{2} \right)\im}\mathop {\lim }\limits_{\varepsilon  \to 0 + } 
R_N^{(Q)} (\varepsilon  + \e^{\frac{\pi }{2}\im} \zeta ,\mu ,\nu )
\\ & - \tfrac12\e^{\left( \left( \nu  + \frac{1}{2} \right)\zeta  + \left( \mu  + \frac{1}{2} \right)\frac{\pi }{2} \right)\im}\mathop {\lim }\limits_{\varepsilon  \to 0 + } R_N^{(Q)} (\varepsilon  + \e^{ - \frac{\pi }{2}\im} \zeta ,\mu ,\nu )
\\ =  & -\tfrac12 \e^{ -\left( \left( \nu  + \frac{1}{2} \right)\zeta  + \left( \mu  + \frac{1}{2} \right)\frac{\pi }{2} \right)\im} \int_0^{ + \infty } s^{\nu  - \frac{1}{2}} \e^{ - s} r_N^{(K)} (s\e^{ - \left( \zeta  - \frac{\pi }{2} \right)\im} \sin \zeta ,\mu )\d s
\\ & - \tfrac12\e^{\left( \left( \nu  + \frac{1}{2} \right)\zeta  + \left( \mu  + \frac{1}{2} \right)\frac{\pi }{2} \right)\im} \int_0^{ + \infty } s^{\nu  - \frac{1}{2}} \e^{ - s} r_N^{(K)} (s\e^{\left( \zeta  - \frac{\pi }{2} \right)\im} \sin \zeta ,\mu )\d s.
\end{split}
\end{gather}
Now, a procedure analogous to \eqref{eq22} yields the estimate \eqref{eq16} for $R_N^{(\mathsf{Q})} (\zeta ,\mu ,\nu )$. 
To prove \eqref{eq23}, we first re-write \eqref{eq24} in the form
\begin{align*}
R_N^{(\mathsf{Q})} (\zeta ,\mu ,\nu ) =\; & \tfrac12\sin (\alpha _{\mu ,\nu ,N} )\int_0^{ + \infty } s^{\nu  - \frac{1}{2}} \e^{ - s} \left( 
\e^{ - \im N\zeta } \im^N r_N^{(K)} (s\e^{ - \left( \zeta  - \frac{\pi }{2} \right)\im } \sin \zeta ,\mu )  \right.\\ 
& \hspace{124.5pt} \left.+ \e^{\im N\zeta } ( - \im)^N r_N^{(K)} (s\e^{\left( \zeta  - \frac{\pi }{2} \right)\im} \sin \zeta ,\mu ) \right)\d s 
\\ & + \tfrac\im2 \cos (\alpha _{\mu ,\nu ,N} )\int_0^{ + \infty } s^{\nu  - \frac{1}{2}} \e^{ - s} \left( \e^{ - \im N\zeta } \im^N 
r_N^{(K)} (s\e^{ - \left( \zeta  - \frac{\pi }{2} \right)\im} \sin \zeta ,\mu )  \right.\\ 
& \hspace{124.5pt} \left.- \e^{\im N\zeta } ( - \im)^N r_N^{(K)} (s\e^{\left( \zeta  - \frac{\pi }{2} \right)\im} \sin \zeta ,\mu ) \right)\d s.
\end{align*}
Suppose that $\Re\nu  > N + \frac{1}{2}$. Next, we employ the relation $r_N^{(K)} (w,\mu ) = a_N (\mu )w^{ - N}  + r_{N + 1}^{(K)} (w,\mu )$ inside 
the second integral, to obtain
\begin{align*}
R_N^{(\mathsf{Q})} (\zeta ,\mu ,\nu ) =\; & \tfrac12\sin (\alpha _{\mu ,\nu ,N} )\int_0^{ + \infty } s^{\nu  - \frac{1}{2}} \e^{ - s} \left( \e^{ - \im N\zeta } \im^N 
r_N^{(K)} (s\e^{ - \left( \zeta  - \frac{\pi }{2} \right)\im } \sin \zeta ,\mu )  \right.\\ 
& \hspace{124.5pt} \left.+ \e^{\im N\zeta } ( - \im)^N r_N^{(K)} (s\e^{\left( \zeta  - \frac{\pi }{2} \right)\im} \sin \zeta ,\mu ) \right)\d s 
\\ & + \tfrac\im2 \cos (\alpha _{\mu ,\nu ,N} )\int_0^{ + \infty } s^{\nu  - \frac{1}{2}} \e^{ - s} \left( \e^{ - \im N\zeta } \im^N 
r_{N+1}^{(K)} (s\e^{ - \left( \zeta  - \frac{\pi }{2} \right)\im} \sin \zeta ,\mu )  \right.\\ 
& \hspace{124.5pt} \left.- \e^{\im N\zeta } ( - \im)^N r_{N+1}^{(K)} (s\e^{\left( \zeta  - \frac{\pi }{2} \right)\im} \sin \zeta ,\mu ) \right)\d s .
\end{align*}
Now, a procedure analogous to \eqref{eq25} gives the bound \eqref{eq23} for $R_N^{(\mathsf{Q})} (\zeta ,\mu ,\nu )$.

\section{Proof of Theorems \ref{thm231} and \ref{thm23}}

In this section, we prove the error bounds stated in Theorems \ref{thm231} and \ref{thm23}. The associated Legendre function of the first kind can be expressed in terms of the hypergeometric function as follows:
\begin{gather}\label{PinF}
\begin{split}
\frac{\sqrt{2\pi\sinh\xi}}{\Gamma(\mu+\nu+1)}P_\nu ^\mu  (\cosh\xi)=\; &
\e^{\pm\left(\frac12-\mu\right)\pi\im-\left(\nu+\frac12\right)\xi}\hyperOlverF{\frac12+\mu}{\frac12-\mu}{\nu+\frac32}{\frac{-\e^{-\xi}}{2\sinh\xi}}\\
&+\e^{\left(\nu+\frac12\right)\xi}\hyperOlverF{\frac12+\mu}{\frac12-\mu}{\nu+\frac32}{\frac{\e^{\xi}}{2\sinh\xi}},
\end{split}
\end{gather}
for $\xi \in \mathcal{D}_1 \setminus (0,+\infty)$. The upper or lower sign is taken in \eqref{PinF} according as $\xi$ is in the upper or lower half-plane (combine \cite[\href{http://dlmf.nist.gov/14.9.E12}{Eq.\ 14.9.12}]{NIST:DLMF} with \eqref{QhypergeoRepr2} and connection formula
\cite[\href{http://dlmf.nist.gov/15.10.E21}{Eq.\ 15.10.21}]{NIST:DLMF}).

Assume that $0<\zeta<\pi$. From \eqref{PinF} and \eqref{Plim}, we deduce
\begin{multline*}
\mathsf{P}_\nu ^\mu  (\cos \zeta ) = \frac{\Gamma (\nu  + \mu  + 1)}{\sqrt{2\pi\sin\zeta}}\left( \e^{ - \left( \left( \nu  + \frac{1}{2} \right)\zeta  + \left( \mu  - \frac{1}{2} \right)\frac{\pi }{2} \right)\im} \hyperOlverF{\frac12+\mu}{\frac12-\mu}{\nu+\frac32}{\frac{ \im \e^{ - \im\zeta } }{2\sin \zeta }} \right.\\ \left.+ \e^{\left( \left( \nu  + \frac{1}{2} \right)\zeta  + \left( \mu  - \frac{1}{2} \right)\frac{\pi }{2} \right)\im} \hyperOlverF{\frac12+\mu}{\frac12-\mu}{\nu+\frac32}{\frac{-\im\e^{\im\zeta } }{2\sin \zeta }} \right)
\end{multline*}
(cf. \cite[p.\ 168]{Magnus1966}). We now substitute the hypergeometric functions by means of the truncated expansion \eqref{hypergeomexp}. In this way, we obtain \eqref{eqfp} with
\begin{gather}\label{eq55}
\begin{split}
\widehat{R}_N^{(\mathsf{P})}(\zeta,\mu,\nu) = \; & \frac{\e^{ - \left( \left( \nu  + \frac{1}{2} \right)\zeta  + \left( \mu  - \frac{1}{2} \right)\frac{\pi }{2} \right)\im}}{2\Gamma \left( \nu  + \frac{3}{2} \right)} R_N^{(F)} \left( \frac{ \im \e^{ - \im\zeta } }{2\sin \zeta },\tfrac{1}{2} + \mu ,\tfrac{1}{2} - \mu ,\nu  + \tfrac{3}{2} \right) \\ & + \frac{\e^{\left( \left( \nu  + \frac{1}{2} \right)\zeta  + \left( \mu  - \frac{1}{2} \right)\frac{\pi }{2} \right)\im}}{2\Gamma \left( \nu  + \frac{3}{2} \right)} R_N^{(F)} \left( \frac{-\im\e^{\im\zeta } }{2\sin \zeta },\tfrac{1}{2} + \mu ,\tfrac{1}{2} - \mu ,\nu  + \tfrac{3}{2} \right) .
\end{split}
\end{gather}
Consequently,
\begin{gather}\label{eq56}
\begin{split}
\left|\widehat{R}_N^{(\mathsf{P})}(\zeta,\mu,\nu)\right| \leq \; & \frac{\e^{\Im (\beta _{\mu ,\nu ,N} )}}{2\left|\Gamma \left( \nu  + \frac{3}{2} \right)\right|}
\left|R_N^{(F)} \left( \frac{ \im \e^{ - \im\zeta } }{2\sin \zeta },\tfrac{1}{2} + \mu ,\tfrac{1}{2} - \mu ,\nu  + \tfrac{3}{2} \right)\right| \\
& + \frac{\e^{- \Im (\beta _{\mu ,\nu ,N} )}}{2\left|\Gamma \left( \nu  + \frac{3}{2} \right)\right|} \left|R_N^{(F)} 
\left( \frac{-\im\e^{\im\zeta } }{2\sin \zeta },\tfrac{1}{2} + \mu ,\tfrac{1}{2} - \mu ,\nu  + \tfrac{3}{2} \right)\right| .
\end{split}
\end{gather}
Now, under the assumptions of Theorem \ref{thm231}, we can apply \eqref{est2} of Theorem \ref{thm22a} to the right-hand side of this inequality to obtain
\begin{align*}
\left|\widehat{R}_N^{(\mathsf{P})}(\zeta,\mu,\nu)\right| \leq\; & \left| \frac{\cos (\pi \mu )}{\cos (\pi \Re \mu )} \right|\left| \frac{\Gamma \left( \frac{1}{2} - \Re \mu  \right)\Gamma (\Re \nu  + \Re \mu  + 1)}{\Gamma \left( \frac{1}{2} - \mu  \right)\Gamma (\nu  + \mu  + 1)}\right| \\ & \times  \frac{\left| a_N (\Re \mu ) \right|\cosh (\Im (\beta _{\mu ,\nu ,N} ))}{\Gamma \left( \Re \nu + \frac{3}{2} + N  \right)\sin ^N \zeta } \times \begin{cases}  \left| \sec \zeta  \right| & \text{ if } \; 2\sin ^2 \zeta  \le 1, \\ 2\sin \zeta  & \text{ if } \; 2\sin ^2 \zeta  > 1.\end{cases}
\end{align*}
Noting that $2\sin ^2 \zeta  \le 1$ occurs whenever $0 < \zeta  \le \frac{\pi}{4}$ or $\frac{3\pi}{4} \le \zeta  < \pi$, and $2\sin ^2 \zeta  > 1$ occurs whenever $\frac{\pi}{4} < \zeta  < \frac{3\pi}{4}$, the proof of the estimate \eqref{boundpq} for $\widehat{R}_N^{(\mathsf{P})}(\zeta,\mu,\nu)$ is complete. With the assumptions of Theorem \ref{thm23}, Theorem \ref{thm22} is applicable to the right-hand side of the inequality \eqref{eq56}, and yields the estimate \eqref{boundpq2} for $\widehat{R}_N^{(\mathsf{P})}(\zeta,\mu,\nu)$.

To prove \eqref{boundp} and \eqref{boundp2}, we first re-write \eqref{eq55} in the form
\begin{align*}
\widehat{R}_N^{(\mathsf{P})}(\zeta,\mu,\nu) = \; & \frac{\cos (\beta _{\mu ,\nu ,N} )}{2\Gamma \left( \nu  + \frac{3}{2} \right)}  \left( \e^{\im\zeta N} \im^N R_N^{(F)} \left( \frac{ \im \e^{ - \im\zeta } }{2\sin \zeta },\tfrac{1}{2} + \mu ,\tfrac{1}{2} - \mu ,\nu  + \tfrac{3}{2} \right) \right. \\ & \hspace{70pt} \left. + \e^{ - \im\zeta N} ( - \im)^N R_N^{(F)} \left( \frac{-\im\e^{\im\zeta } }{2\sin \zeta },\tfrac{1}{2} + \mu ,\tfrac{1}{2} - \mu ,\nu  + \tfrac{3}{2} \right) \right) \\ &
 - \frac{\im \sin (\beta _{\mu ,\nu ,N} )}{2\Gamma \left( \nu  + \frac{3}{2} \right)}  \left( \e^{\im\zeta N} \im^N R_N^{(F)} \left( \frac{ \im \e^{ - \im\zeta } }{2\sin \zeta },\tfrac{1}{2} + \mu ,\tfrac{1}{2} - \mu ,\nu  + \tfrac{3}{2} \right) \right.  \\ & \hspace{70pt} \left.- \e^{ - \im\zeta N} ( - \im)^N R_N^{(F)} \left( \frac{-\im\e^{\im\zeta } }{2\sin \zeta },\tfrac{1}{2} + \mu ,\tfrac{1}{2} - \mu ,\nu  + \tfrac{3}{2} \right) \right).
\end{align*}
Next, we employ the relation
\begin{equation}\label{eq61}
R_N^{(F)} \left( w,\tfrac{1}{2} + \mu ,\tfrac{1}{2} - \mu ,\nu  + \tfrac{3}{2} \right) = ( - 2)^N \frac{a_N (\mu )}{\left( \nu  + \frac{3}{2} \right)_N }w^N  + R_{N + 1}^{(F)} \left( w,\tfrac{1}{2} + \mu ,\tfrac{1}{2} - \mu ,\nu  + \tfrac{3}{2} \right)
\end{equation}
inside the second large parenthesis, to deduce
\begin{align*}
\widehat{R}_N^{(\mathsf{P})}(\zeta,\mu,\nu) = \; & \frac{\cos (\beta _{\mu ,\nu ,N} )}{2\Gamma \left( \nu  + \frac{3}{2} \right)}  \left( \e^{\im\zeta N} \im^N R_N^{(F)} \left( \frac{ \im \e^{ - \im\zeta } }{2\sin \zeta },\tfrac{1}{2} + \mu ,\tfrac{1}{2} - \mu ,\nu  + \tfrac{3}{2} \right) \right. \\ & \hspace{100pt} \left. + \e^{ - \im\zeta N} ( - \im)^N R_N^{(F)} \left( \frac{-\im\e^{\im\zeta } }{2\sin \zeta },\tfrac{1}{2} + \mu ,\tfrac{1}{2} - \mu ,\nu  + \tfrac{3}{2} \right) \right) \\ &
 - \frac{\im \sin (\beta _{\mu ,\nu ,N} )}{2\Gamma \left( \nu  + \frac{3}{2} \right)}  \left( \e^{\im\zeta N} \im^N R_{N+1}^{(F)} \left( \frac{ \im \e^{ - \im\zeta } }{2\sin \zeta },\tfrac{1}{2} + \mu ,\tfrac{1}{2} - \mu ,\nu  + \tfrac{3}{2} \right) \right.  \\ & \hspace{100pt} \left.- \e^{ - \im\zeta N} ( - \im)^N R_{N+1}^{(F)} \left( \frac{-\im\e^{\im\zeta } }{2\sin \zeta },\tfrac{1}{2} + \mu ,\tfrac{1}{2} - \mu ,\nu  + \tfrac{3}{2} \right) \right).
\end{align*}
Thus,
\begin{gather}\label{eq57}
\begin{split}
\left|\widehat{R}_N^{(\mathsf{P})}(\zeta,\mu,\nu)\right| \leq \; & \frac{\left|\cos (\beta _{\mu ,\nu ,N} )\right|}{2\left|\Gamma \left( \nu  + \frac{3}{2} \right)\right|}  \left( \left| R_N^{(F)} \left( \frac{ \im \e^{ - \im\zeta } }{2\sin \zeta },\tfrac{1}{2} + \mu ,\tfrac{1}{2} - \mu ,\nu  + \tfrac{3}{2} \right)\right| \right. \\ & \hspace{120pt} \left. + \left| R_N^{(F)} \left( \frac{-\im\e^{\im\zeta } }{2\sin \zeta },\tfrac{1}{2} + \mu ,\tfrac{1}{2} - \mu ,\nu  + \tfrac{3}{2} \right) \right|\right) \\ &
 + \frac{\left|\sin (\beta _{\mu ,\nu ,N} )\right|}{2\left|\Gamma \left( \nu  + \frac{3}{2} \right)\right|}  \left( \left| R_{N+1}^{(F)} \left( \frac{ \im \e^{ - \im\zeta } }{2\sin \zeta },\tfrac{1}{2} + \mu ,\tfrac{1}{2} - \mu ,\nu  + \tfrac{3}{2} \right) \right|\right.  \\ & \hspace{120pt} \left.+\left| R_{N+1}^{(F)} \left( \frac{-\im\e^{\im\zeta } }{2\sin \zeta },\tfrac{1}{2} + \mu ,\tfrac{1}{2} - \mu ,\nu  + \tfrac{3}{2} \right)\right| \right).
\end{split}
\end{gather}
Under the assumptions of Theorem \ref{thm231}, we can apply \eqref{est2} of Theorem \ref{thm22a} to the right-hand side of this inequality to obtain the bound \eqref{boundp}. With the assumptions of Theorem \ref{thm23}, Theorem \ref{thm22} is applicable to the right-hand side of the inequality \eqref{eq57}, leading to the estimate \eqref{boundp2}.

We conclude this section with the proof of the estimates for the reminder term $\widehat{R}_N^{(\mathsf{Q})}(\zeta,\mu,\nu)$. From \eqref{QhypergeoRepr2} and \eqref{eq20}, we deduce
\begin{multline*}
\mathsf{Q}_\nu ^\mu  (\cos \zeta ) = -\frac{\im}{2}\sqrt {\frac{\pi }{2\sin \zeta }} \Gamma (\nu  + \mu  + 1)\left( \e^{ - \left( \left( \nu  + \frac{1}{2} \right)\zeta  + \left( \mu  - \frac{1}{2} \right)\frac{\pi }{2} \right)\im} \hyperOlverF{\frac12+\mu}{\frac12-\mu}{\nu+\frac32}{\frac{ \im \e^{ - \im\zeta } }{2\sin \zeta }} \right.\\ \left.- \e^{\left( \left( \nu  + \frac{1}{2} \right)\zeta  + \left( \mu  - \frac{1}{2} \right)\frac{\pi }{2} \right)\im} \hyperOlverF{\frac12+\mu}{\frac12-\mu}{\nu+\frac32}{\frac{-\im\e^{\im\zeta } }{2\sin \zeta }} \right) 
\end{multline*}
(cf. \cite[p.\ 168]{Magnus1966}). We now substitute the hypergeometric functions by means of the truncated expansion \eqref{hypergeomexp}. In this way, we obtain \eqref{eqfq} with
\begin{gather}\label{eq60}
\begin{split}
\widehat{R}_N^{(\mathsf{Q})}(\zeta,\mu,\nu) = \; & \frac{\im\e^{ - \left( \left( \nu  + \frac{1}{2} \right)\zeta  + \left( \mu  - \frac{1}{2} \right)\frac{\pi }{2} \right)\im}}{2\Gamma \left( \nu  + \frac{3}{2} \right)} R_N^{(F)} \left( \frac{ \im \e^{ - \im\zeta } }{2\sin \zeta },\tfrac{1}{2} + \mu ,\tfrac{1}{2} - \mu ,\nu  + \tfrac{3}{2} \right) \\ & - \frac{\im\e^{\left( \left( \nu  + \frac{1}{2} \right)\zeta  + \left( \mu  - \frac{1}{2} \right)\frac{\pi }{2} \right)\im}}{2\Gamma \left( \nu  + \frac{3}{2} \right)} R_N^{(F)} \left( \frac{-\im\e^{\im\zeta } }{2\sin \zeta },\tfrac{1}{2} + \mu ,\tfrac{1}{2} - \mu ,\nu  + \tfrac{3}{2} \right).
\end{split}
\end{gather}
We can now proceed analogously to the case of $\widehat{R}_N^{(\mathsf{P})}(\zeta,\mu,\nu)$ and derive the bounds \eqref{boundpq} and \eqref{boundpq2} for $\widehat{R}_N^{(\mathsf{Q})}(\zeta,\mu,\nu)$. To prove \eqref{boundq} and \eqref{boundq2}, we first re-write \eqref{eq60} in the form
\begin{align*}
\widehat{R}_N^{(\mathsf{Q})}(\zeta,\mu,\nu) = \; & \frac{\sin (\beta _{\mu ,\nu ,N} )}{2\Gamma \left( \nu  + \frac{3}{2} \right)}  \left( \e^{\im\zeta N} \im^N R_N^{(F)} \left( \frac{ \im \e^{ - \im\zeta } }{2\sin \zeta },\tfrac{1}{2} + \mu ,\tfrac{1}{2} - \mu ,\nu  + \tfrac{3}{2} \right) \right. \\ & \hspace{70pt} \left. + \e^{ - \im\zeta N} ( - \im)^N R_N^{(F)} \left( \frac{-\im\e^{\im\zeta } }{2\sin \zeta },\tfrac{1}{2} + \mu ,\tfrac{1}{2} - \mu ,\nu  + \tfrac{3}{2} \right) \right) \\ &
+ \frac{\im \cos (\beta _{\mu ,\nu ,N} )}{2\Gamma \left( \nu  + \frac{3}{2} \right)}  \left( \e^{\im\zeta N} \im^N R_N^{(F)} \left( \frac{ \im \e^{ - \im\zeta } }{2\sin \zeta },\tfrac{1}{2} + \mu ,\tfrac{1}{2} - \mu ,\nu  + \tfrac{3}{2} \right) \right.  \\ & \hspace{70pt} \left.- \e^{ - \im\zeta N} ( - \im)^N R_N^{(F)} \left( \frac{-\im\e^{\im\zeta } }{2\sin \zeta },\tfrac{1}{2} + \mu ,\tfrac{1}{2} - \mu ,\nu  + \tfrac{3}{2} \right) \right).
\end{align*}
Next, we employ the relation \eqref{eq61} inside the second large parenthesis, to deduce
\begin{align*}
\widehat{R}_N^{(\mathsf{Q})}(\zeta,\mu,\nu) = \; & \frac{\sin (\beta _{\mu ,\nu ,N} )}{2\Gamma \left( \nu  + \frac{3}{2} \right)}  \left( \e^{\im\zeta N} \im^N R_N^{(F)} \left( \frac{ \im \e^{ - \im\zeta } }{2\sin \zeta },\tfrac{1}{2} + \mu ,\tfrac{1}{2} - \mu ,\nu  + \tfrac{3}{2} \right) \right. \\ & \hspace{70pt} \left. + \e^{ - \im\zeta N} ( - \im)^N R_N^{(F)} \left( \frac{-\im\e^{\im\zeta } }{2\sin \zeta },\tfrac{1}{2} + \mu ,\tfrac{1}{2} - \mu ,\nu  + \tfrac{3}{2} \right) \right) \\ &
+ \frac{\im \cos (\beta _{\mu ,\nu ,N} )}{2\Gamma \left( \nu  + \frac{3}{2} \right)}  \left( \e^{\im\zeta N} \im^N R_{N+1}^{(F)} \left( \frac{ \im \e^{ - \im\zeta } }{2\sin \zeta },\tfrac{1}{2} + \mu ,\tfrac{1}{2} - \mu ,\nu  + \tfrac{3}{2} \right) \right.  \\ & \hspace{70pt} \left.- \e^{ - \im\zeta N} ( - \im)^N R_{N+1}^{(F)} \left( \frac{-\im\e^{\im\zeta } }{2\sin \zeta },\tfrac{1}{2} + \mu ,\tfrac{1}{2} - \mu ,\nu  + \tfrac{3}{2} \right) \right).
\end{align*}
We can now proceed similarly to the case of $\widehat{R}_N^{(\mathsf{P})}(\zeta,\mu,\nu)$ and derive the bounds \eqref{boundq} and \eqref{boundq2} for $\widehat{R}_N^{(\mathsf{Q})}(\zeta,\mu,\nu)$.

\section{Ferrers functions for large \texorpdfstring{$\mu$}{mu} and fixed \texorpdfstring{$\nu$}{nu}: proof}\label{Ferrerslargemu}

To obtain large-$\mu$ asymptotic expansions for the Ferrers functions, we can combine the expansions \eqref{formal1} and \eqref{formal2} with the identities
\begin{align*}
\mathsf{P}_{\nu} ^{\mu}  (\cos \zeta ) = \; & 
\frac{\e^{\frac{2\mu-4\nu-3}{4}\pi \im}}{\Gamma(\nu-\mu+1)}\sqrt{\frac2{\pi\sin\zeta}} Q_{-\mu-\frac12} ^{\nu+\frac12}  (-\im\cot\zeta)\\
=\; &\sin(\pi\mu)\e^{\frac{2\mu-1}{4}\pi \im}\Gamma(\mu-\nu)\sqrt{\frac2{\pi\sin\zeta}} P_{\mu-\frac12} ^{\nu+\frac12}  (-\im\cot\zeta)\\
& -\frac{\sin(\pi(\nu+\mu))}{\pi}\e^{\frac{2\mu-4\nu-3}{4}\pi \im}\Gamma(\mu-\nu)\sqrt{\frac2{\pi\sin\zeta}} Q_{\mu-\frac12} ^{\nu+\frac12}(-\im\cot\zeta),
\end{align*}
\begin{align*}
\mathsf{Q}_{\nu} ^{\mu}  (\cos \zeta ) = \; &\Gamma(\nu+\mu+1)\sqrt{\frac{\pi}{8\sin\zeta}}
\left(\e^{-\frac{2\mu+1}{4}\pi\im}P_{\mu-\frac12} ^{-\nu-\frac12}  (-\im\cot\zeta)
+\e^{\frac{2\mu+1}{4}\pi\im}P_{\mu-\frac12} ^{-\nu-\frac12}  (\im\cot\zeta)\right)\\
= \; &\Gamma(\mu-\nu)\sqrt{\frac{\pi}{8\sin\zeta}}
\left(\e^{-\frac{2\mu+1}{4}\pi\im}P_{\mu-\frac12} ^{\nu+\frac12}  (-\im\cot\zeta)
+\e^{\frac{2\mu+1}{4}\pi\im}P_{\mu-\frac12} ^{\nu+\frac12}  (\im\cot\zeta)\right)\\
&+\Gamma(\mu-\nu)\cos(\pi\nu)\frac{\e^{-(\nu+\frac12)\pi\im}}{\sqrt{2\pi\sin\zeta}}
\left(\e^{\frac{3-2\mu}{4}\pi\im}Q_{\mu-\frac12} ^{\nu+\frac12}  (-\im\cot\zeta)
+\e^{\frac{2\mu-3}{4}\pi\im}Q_{\mu-\frac12} ^{\nu+\frac12}  (\im\cot\zeta)\right).
\end{align*}
These identities follow by combining \eqref{eq14} and \eqref{eq59} with Whipple's formulae \eqref{Whipple1} and \eqref{Whipple2} and
then using the connection formulae \cite[\href{http://dlmf.nist.gov/14.9.E12}{Eq.\ 14.9.12}]{NIST:DLMF} and \cite[\href{http://dlmf.nist.gov/14.9.E15}{Eq.\ 14.9.15}]{NIST:DLMF}.

Let $\eps$ be an arbitrary fixed positive number. Since $\pm \im\cot \zeta  = \cosh \left( \log \left(\cot \left( \tfrac{1}{2}\zeta \right)\right) \pm \frac{\pi}{2}\im \right)$, it follows from the expansions \eqref{formal1} and \eqref{formal2} that for $\eps\leq\zeta<\frac{\pi}{2}$, the asymptotic expansions \eqref{formal12} and \eqref{formal13} hold as $\Re \mu\to+\infty$, with $\Im \mu$ being bounded. Examination of the remainder terms in the limiting case $\zeta\to \frac{\pi}{2}$ (using Theorems \ref{thm1} and \ref{thm2}), shows that these inverse factorial expansions are actually valid for $\eps\leq\zeta\leq \frac{\pi}{2}$. To extend these expansions to the larger interval $\eps\leq\zeta\leq \pi-\eps$, we can proceed as follows. The Ferrers functions satisfy the connection relations
\begin{align*}
\mathsf{P}_{\nu} ^{\mu}  (-x) & = \cos(\pi(\nu+\mu))\mathsf{P}_{\nu} ^{\mu}  (x)-\tfrac{2}{\pi}\sin(\pi(\nu+\mu))\mathsf{Q}_{\nu} ^{\mu}  (x),\\
\mathsf{Q}_{\nu} ^{\mu}  (-x) & =  -\cos(\pi(\nu+\mu))\mathsf{Q}_{\nu} ^{\mu}  (x)-\tfrac{\pi}{2}\sin(\pi(\nu+\mu))\mathsf{P}_{\nu} ^{\mu}  (x),
\end{align*}
for $0\leq x<1$ (see, for example, \cite[\href{http://dlmf.nist.gov/14.9.E8}{Eq.\ 14.9.8}]{NIST:DLMF} and \cite[\href{http://dlmf.nist.gov/14.9.E10}{Eq.\ 14.9.10}]{NIST:DLMF}). We employ these relations with $x = -\cos\zeta = \cos (\pi -\zeta )$, $\frac{\pi}{2}<\zeta\leq \pi-\eps$, and use the previously established asymptotic expansions for the right-hand sides of the resulting equalities. After simplification, we find that $\mathsf{P}_{\nu} ^{\mu}  (\cos \zeta ) $ and $\mathsf{Q}_{\nu} ^{\mu}  (\cos \zeta )$, $\frac{\pi}{2}< \zeta\leq \pi-\eps$, possess the same asymptotic expansions as they do for $\eps\leq\zeta\leq \frac{\pi}{2}$.

\section{Numerical examples}\label{Numerics}

In this section, we provide some numerical examples to demonstrate the sharpness of our error bounds and the accuracy of our asymptotic approximations. Taking $\nu=10$, $\mu=\frac65$ and $\xi=\frac45$, 
we obtain $\e^{-\pi \im\mu }Q_\nu ^\mu  (\cosh \xi )\approx 0.00168049$. If we take $N=2$ terms on the right-hand side (RHS) of \eqref{formal2},
we obtain the approximation $0.00169164$, that is, in Theorem \ref{thm2} we have $R_N^{(Q)} (\xi ,\mu ,\nu)\approx 8642.4139$, and the RHS
of \eqref{eq10a} is approximately $10618.854$. 
Note that $8642/10619\approx0.81$, which is close to $1$, demonstrating the sharpness of the error bound \eqref{eq10a}.

\begin{figure}[htbp]
\centering
\includegraphics[width=0.4\textwidth]{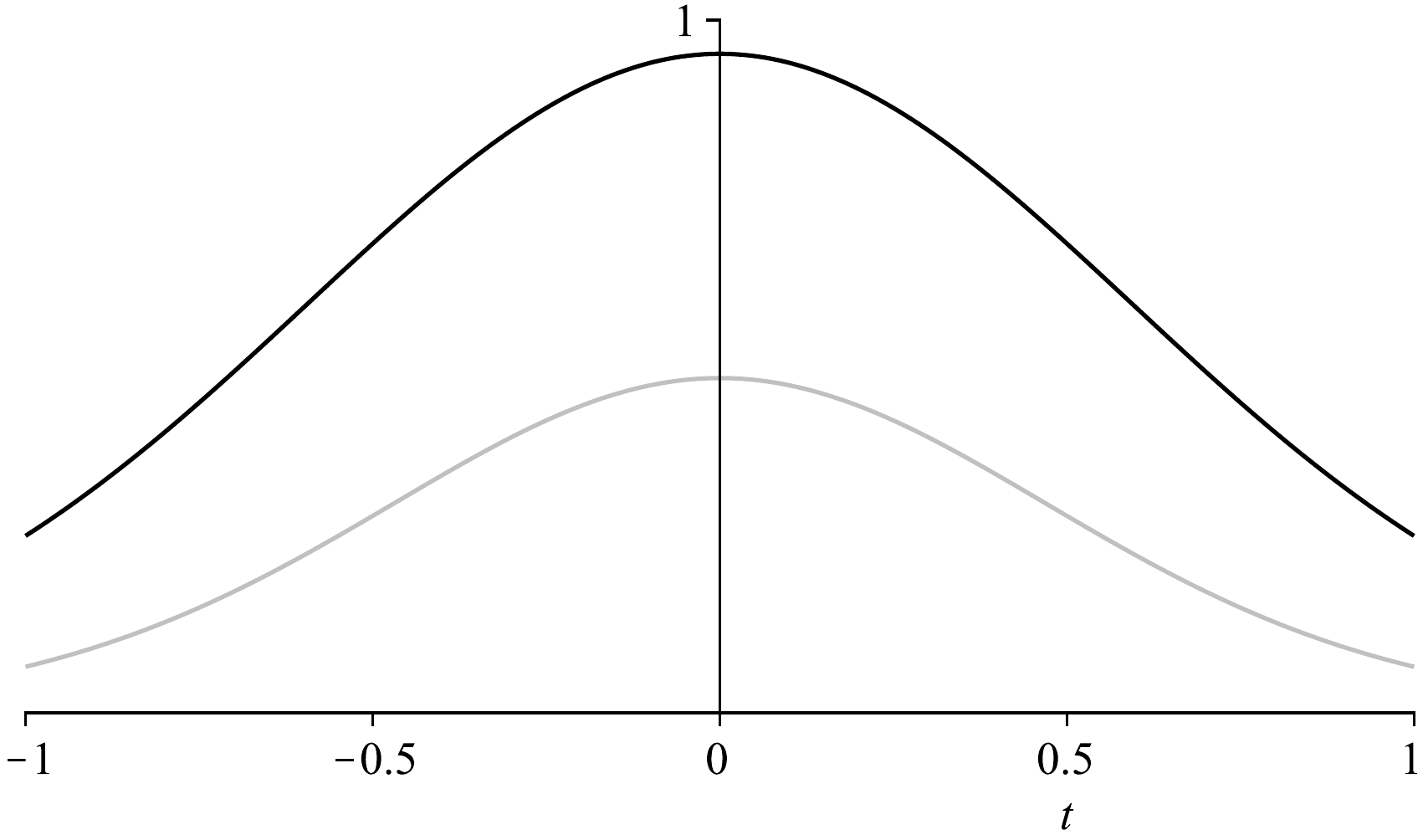}%
\caption{\rm (LHS of \eqref{eq10a})/(RHS of \eqref{eq10a}) for the cases $\nu=10+5\im t$, $\mu=\frac65$, $\xi=\frac45$ and $N=1$ (black) and $N=5$ (grey).}
\label{fig2}
\end{figure}

In Figure \ref{fig2}, we take $\nu=10+5\im t$, $\mu=\frac65$ and $\xi=\frac45$, and plot (LHS of \eqref{eq10a})/(RHS of \eqref{eq10a}) in the cases that
$N=1$ (black) and $N=5$ (grey), respectively. The graph illustrates that our error bound \eqref{eq10a} is sharp, and that the result in Theorem \ref{thm2} is indeed useful when $\Re\nu\to+\infty$.

\begin{figure}[htbp]
\centering
\includegraphics[width=0.4\textwidth]{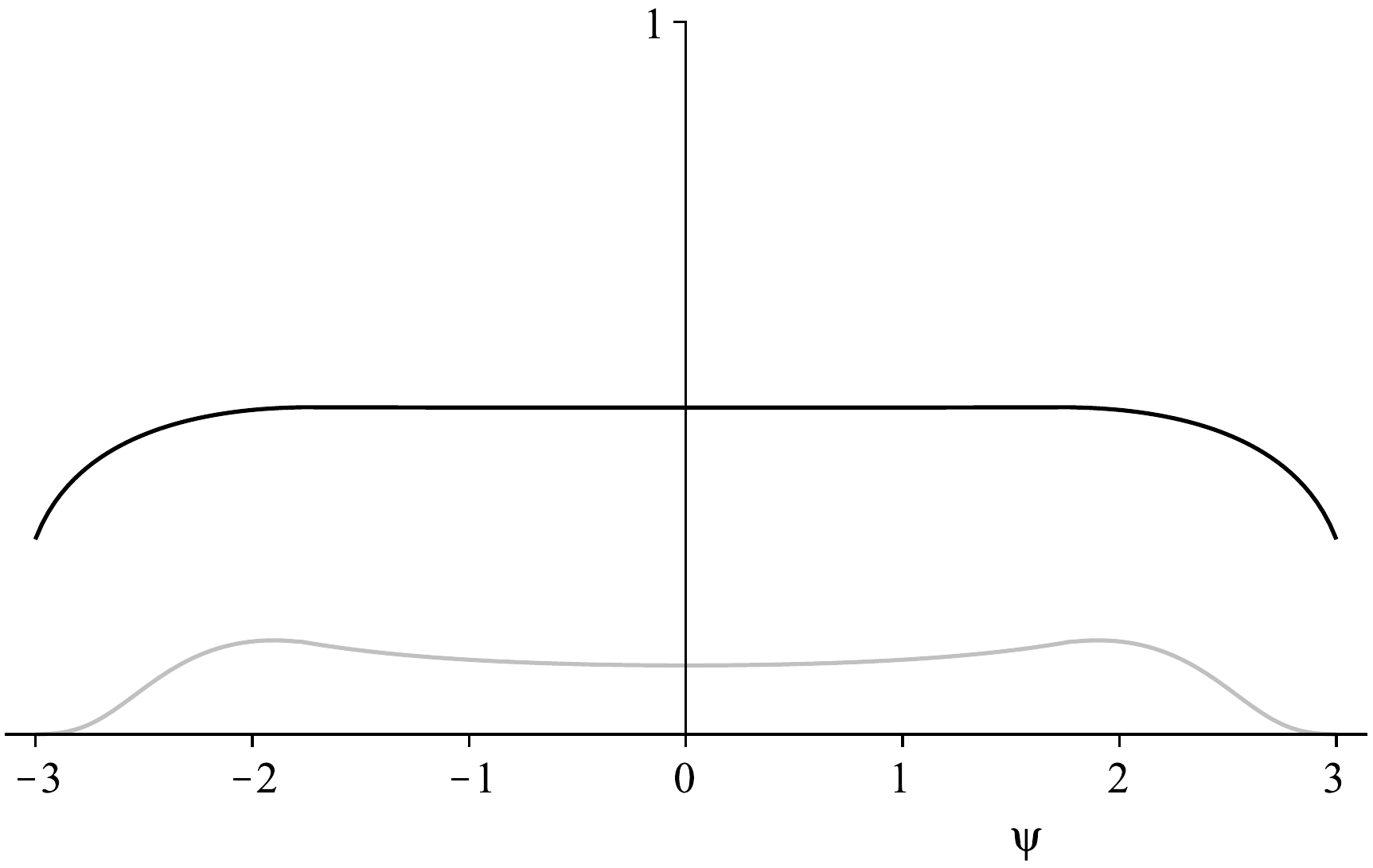}%
\caption{\rm (LHS of \eqref{factexpQa})/(RHS of \eqref{factexpQa}) for the cases $\nu=10\e^{\im \psi}$, $\mu=\frac65$, $\xi=\frac45$ and $N=1$ (black) and $N=5$ (grey).}
\label{fig3}
\end{figure}

In Figure \ref{fig3}, we take $\nu=10\e^{\im\psi}$, $\mu=\frac65$ and $\xi=\frac45$, and plot (LHS of \eqref{factexpQa})/(RHS of \eqref{factexpQa}) in the cases that
$N=1$ (black) and $N=5$ (grey), respectively. The graph illustrates that our error bound \eqref{factexpQa} is reasonably sharp, and that the result in Theorem \ref{thm21}
holds when $|\nu|\to+\infty$ in large sectors of the form $|\arg\nu|\leq \pi-\delta$ ($<\pi$).

\begin{figure}[htbp]
\centering
\includegraphics[width=0.4\textwidth]{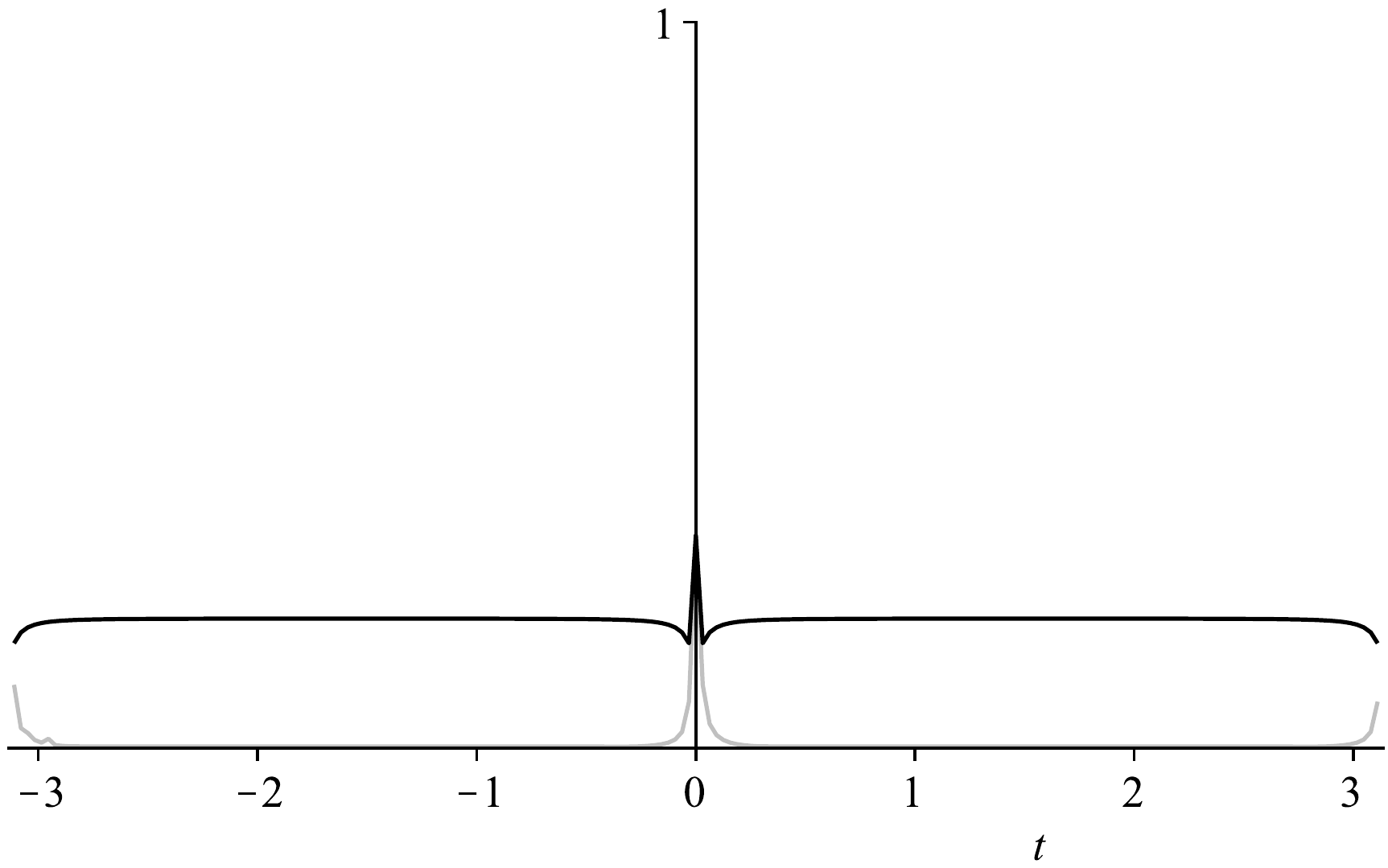}%
\caption{\rm (LHS of \eqref{factexpQa})/(RHS of \eqref{factexpQa}) for the cases $\mu=\frac65$, $\xi=0.01+\im t$, $N=1$ and $\nu=10$ (black) and $\nu=10\e^{2\im}$ (grey).}
\label{fig4}
\end{figure}

In Figure \ref{fig4}, we take $\mu=\frac65$, $\xi=0.01+\im t$ and $N=1$, and plot (LHS of \eqref{factexpQa})/(RHS of \eqref{factexpQa}) in the cases that
$\nu=10$ (black) and $\nu=10\e^{2\im}$ (grey). These are extreme cases. We observe that when $\nu=10$, the error bound is still sharp. Even in the case that $\nu=10\e^{2\im}$, we see that the bound is realistic near the endpoints $\xi=0,\pm\pi\im$.

In Figure \ref{fig5}, we illustrate how the accuracy of the large-$\nu$ approximations \eqref{eq3} and \eqref{eq10} for the associated Legendre functions $P_\nu ^\mu  (z )$
and  $Q_\nu ^\mu  (z )$ diminishes as $z$ approaches $1$, but they are uniformly valid for $z\geq 1+\eps$ ($>1$).
Likewise, in Figure \ref{fig6}, we illustrate how the accuracy of the large-$\mu$ approximation \eqref{formal5} for the associated Legendre function $Q_\nu ^\mu  (z )$
decreases as $z$ approaches $+\infty$, but is uniformly valid for $z\in(1,Z]$.
Similarly, in Table \ref{table1}, we illustrate how the accuracy of the large-$\nu$ approximations \eqref{eq15} and \eqref{eq21} for the Ferrers functions 
$\mathsf{P}_\nu ^\mu  (x)$ and  $\mathsf{Q}_\nu ^\mu  (x)$ diminishes as $x$ approaches $1$.

\begin{figure}[htbp]
\centering
\includegraphics[width=0.45\textwidth]{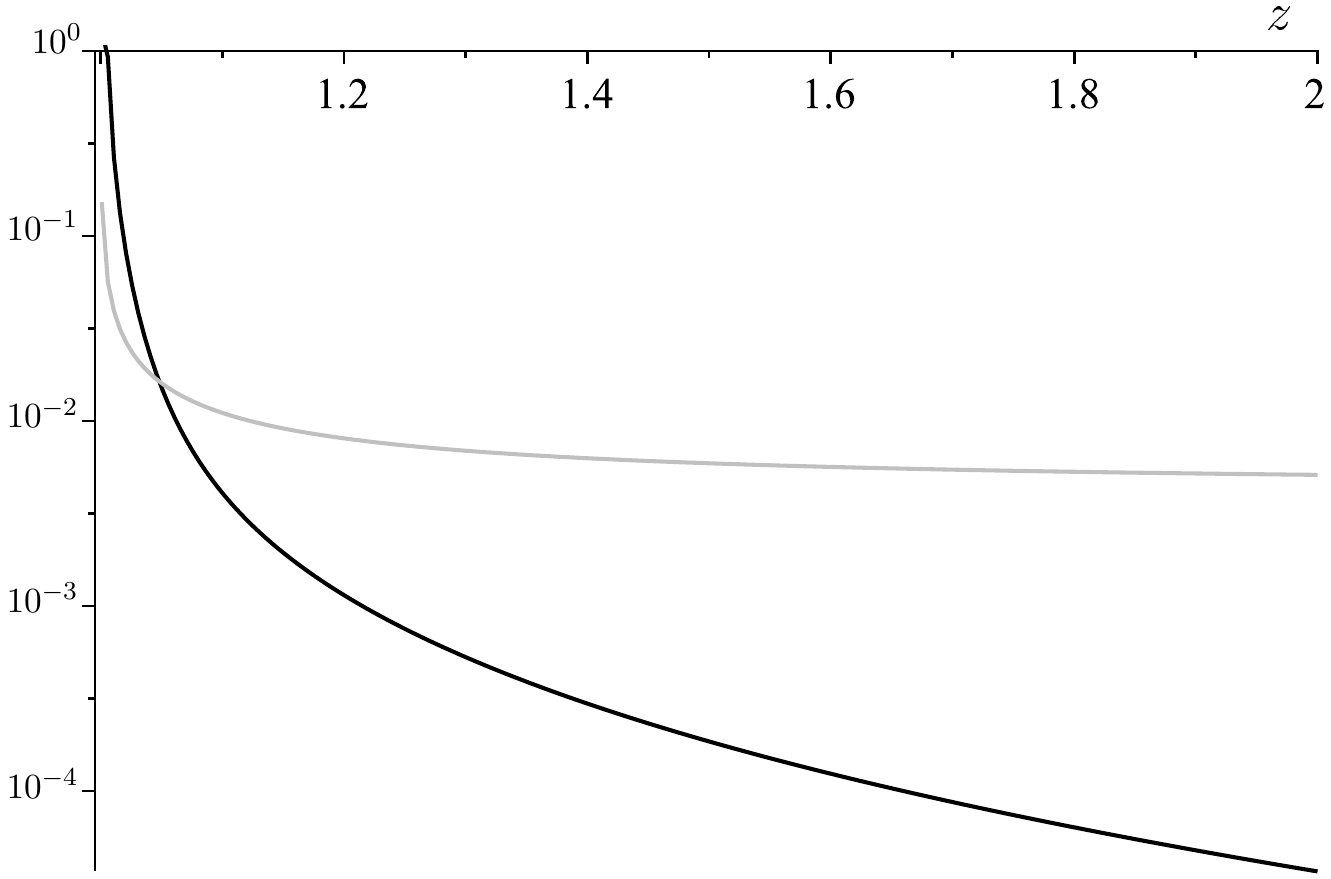}%
\hglue 1cm\includegraphics[width=0.45\textwidth]{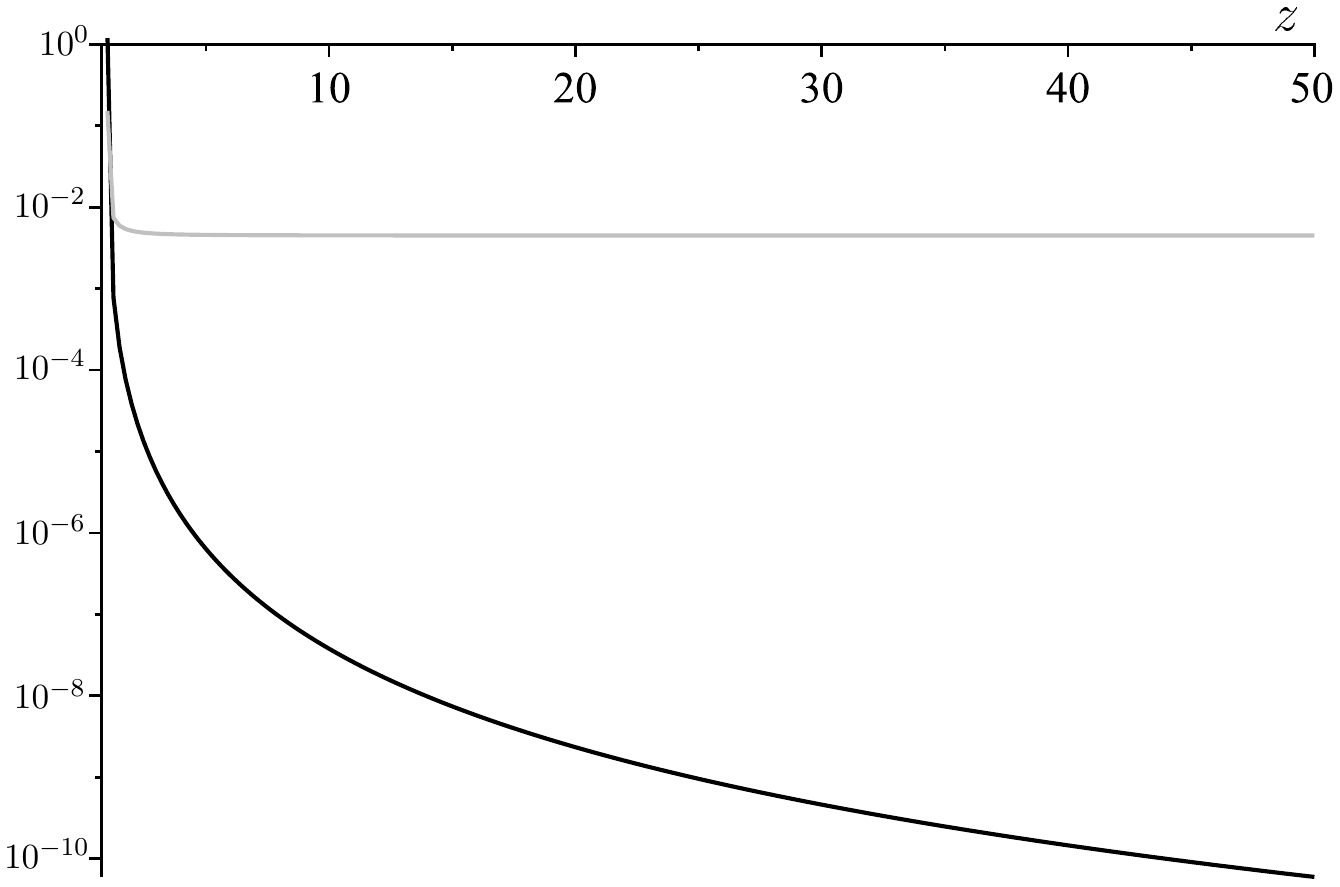}%
\caption{\rm Relative errors of the approximations \eqref{eq3} (black) and \eqref{eq10} (grey) for the cases $\nu=10$, 
$\mu=\frac65$, $N=M=2$, $z=\cosh\xi\in(1,2)$ (left) and $z=\cosh\xi\in(1,50)$ (right).}
\label{fig5}
\end{figure}

\begin{figure}[htbp]
\centering
\includegraphics[width=0.45\textwidth]{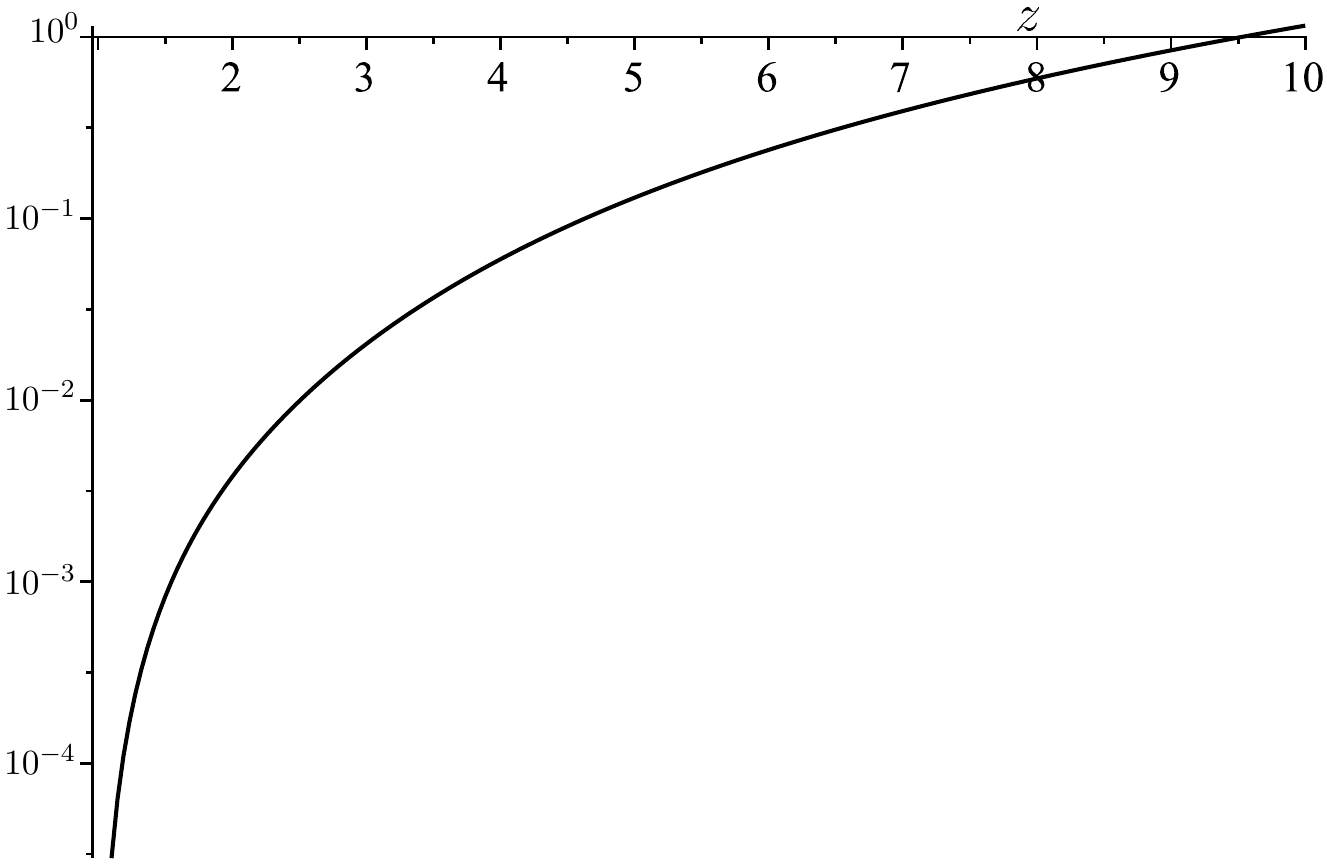}%
\caption{\rm Relative error of the approximation \eqref{formal5} for the case 
$\nu=\frac65$, $\mu=10$, $z=\coth\xi\in(1,10)$ and taking 2 terms in both sums on the right-hand side of \eqref{formal5}.}
\label{fig6}
\end{figure}

\begin{table}[h]
\renewcommand{\arraystretch}{1.25}
\begin{center}
\begin{tabular}{r|r|r|r|r|r|r}
\hline\hline
\multicolumn{1}{c|}{$x=\cos\zeta$} & \multicolumn{1}{|c|}{$\mathsf{P}_\nu ^\mu  (\cos \zeta )$} & \multicolumn{1}{|c|}{RHS of \eqref{eq15}} & \multicolumn{1}{|c|}{relative error} & \multicolumn{1}{|c|}{$\mathsf{Q}_\nu ^\mu  (\cos \zeta )$} & \multicolumn{1}{|c|}{RHS of \eqref{eq21}} & \multicolumn{1}{|c}{relative error} \\
\hline
$0.9$ & $-5.58976$ & $-5.57356$ & $0.00291$ & $13.1050$ & $13.1067$ & $0.00013$ \\
$0.95$ & $1.86065$ & $1.82948$ & $0.01704$ & $-18.4721$ & $-18.4407$ & $0.00170$ \\
$0.99$ & $-8.96386$ & $-8.72156$ & $0.02778$ & $24.8288$ & $24.7007$ & $0.00519$ \\
$0.999$ & $-31.5013$ & $-32.1480$ & $0.02012$ & $-35.7525$ & $-30.3175$ & $0.17927$ \\
\hline\hline
\end{tabular}
\end{center}
\caption{Relative errors of the approximations \eqref{eq15} and \eqref{eq21} for the case $\nu=20$, $\mu=\frac65$ and $N=2$.}
\label{table1}
\end{table}%

\section*{Acknowledgement}
The authors' research was supported by a research grant (GRANT11863412/70NANB15H221) from the National Institute of Standards and Technology. The first author was also supported by a Premium Postdoctoral Fellowship of the Hungarian Academy of Sciences. The authors thank the referees for very helpful comments and suggestions for improving the presentation.


\bigskip

\end{document}